\documentclass[a4paper,12pt]{article}
\usepackage[latin1]{inputenc}
\usepackage{amssymb}
\usepackage{amsmath}
\usepackage[english]{babel}
\usepackage{stmaryrd,dsfont,hyperref}
\usepackage{url} 
\DeclareMathOperator*{\esssup}{ess\,sup}
\DeclareMathOperator*{\essinf}{ess\,inf}

\newtheorem{theorem}{Theorem}
\newtheorem{remark}{Remark}
\newtheorem{lemma}{Lemma}
\newtheorem{corollary}{Corollary}
\newtheorem{proposition}{Proposition}
\newtheorem{definition}{Definition}
\newenvironment{proof}[1][Proof]{\textbf{#1.} }{\ \rule{0.5em}{0.5em}}
\usepackage{vmargin}
   \setmarginsrb{3.5cm}{4cm}{3.5cm}{4cm}{0cm}{0cm}{0cm}{0.5cm}
\title{Well-posedness of reflected BSDEs with default time and irregular barrier: An application to optimal control}
\author{
Badr ELMANSOURI\textsuperscript{*,a} and Mohamed EL OTMANI\textsuperscript{b}\\\\
\textsuperscript{a} Cadi Ayyad University (UCA) \\
National School of Applied Sciences of Marrakech (ENSA-M)\\
BP 575, Avenue Abdelkrim Khattabi, 40000, Guéliz, Marrakech, Morocco\\\\
\textsuperscript{b} Laboratory of Analysis and Applied Mathematics (LAMA) \\
Faculty of Sciences Agadir, Ibn Zohr University\\
BP 8106, Hay Dakhla, 80000,
Agadir, Morocco\\\\
Emails: \url{b.elmansouri@uca.ac.ma} \& \url{m.elotmani@uiz.ac.ma} \\
\textsuperscript{*} Corresponding author.
}

\date{}
\begin{document}
\maketitle
\begin{abstract}
We consider a reflected backward stochastic differential equations with default time and an optional  barrier in a filtration generated by a one-dimensional Brownian motion and a defaultable process. We suppose that the barrier have trajectories with left and right finite limits. We provide the existence and uniqueness result when the coefficient is scholastic Lipschitz by using a modified penalization method.  Under an additional assumption of right-upper semi-continuity along stopping times on the trajectories of the barrier, we characterize the state process for such RBSDEs as the value function of an optimal stopping problem associated with a non-linear $f$-expectation.
\end{abstract}
\textbf{{keyword :}} Reflected backward stochastic differential equations, optional barriers, Default time, Penalization method, $f$-expectation
%
\section{Introduction}
The notion of backward stochastic differential equations (BSDEs, for short) have been introduced by Bismut \cite{bismut1973conjugate} in the linear case, then extended to the general case of non-linear driver by Pardoux and Peng \cite{pardoux1990adapted}. For such of equations, European option pricing and hedging are only two of the many financial uses for BSDEs (see for instance \cite{barrieu2005inf, el1997backward}). 

Reflected backward stochastic differential equations  (RBSDEs, for short) have been introduced for the first time by El Karoui et al. \cite{el1997reflected}. Those equations are a specific kind of BSDEs in which the initial part of the solution is restricted to remain above a certain barrier. In their seminal work \cite{el1997reflected}, the authors deal with a Brownian framework, a Lipschitz driver and a continuous obstacle. Since this foundation, many efforts have been made to generalize this work to the case of discontinuities by dealing with barriers having right-continuous and left limits trajectories or a more general filtration additionally including Poisson or L\'{e}vy processes (see for instance \cite{STE, ElKaroui1997, ElOtmani2009, essaky2008reflected, hamadene2002reflected, HO, hamadene2016reflected} among others). Note that those equations are highly motivated by the pricing of American contingent claims in different financial markets (we also refer to \cite{elmansouri2026rcll}).

Going beyond right continuity, the classical I\^o's formula and Doob-Meyer decomposition are no longer available to use, and some other generalized results from the general theory of stochastic processes and optimal stopping field need to be introduced and employed (see for example Mertens' decomposition for not necessarily right-continuous supermartingales \cite[Theorem A.1]{Grigorova2017}, Gal'chouk \cite{gal1981optional} and Lenglart \cite{Leng} for a generalization of the classical  I\^o's formula,  I\^o's formula for processes with regulated trajectories \cite[Theorem A.1]{klimsiak2019reflected}, and the references therein to cite a few). RBSDEs with a not necessary right-continuous obstacle has been an interesting topic during last dedicates due to it's link with optimization problems associated with the so called {\it non linear expectation}, $f$-expectation or $f$-evaluation operator defined through the solution of a given classical BSDE which have been used in the literature on dynamic risk measures (see or example \cite{gianin2006risk,quenez2013bsdes}). In this context, and to the best of our knowledge, Grigorova et al. \cite{Grigorova2017}  was the first paper dealing with RBSDEs
with barriers that are not right-continuous. With a square integrable barrier having left-limits and right-upper semi-continuous trajectories, and a Lipschitz continuity assumption on the driver, the authors in \cite{Grigorova2017} proofs the existence and uniqueness result using the alternative tools mentioned above beyond the right-continuous context. Moreover, the authors also provides some links with an  optimal stopping with $f$-expectations. Following this work, many efforts have been made to study this type of RBSDEs in a more general context under weaker assumptions on the data (see for instance \cite{baadi2018reflected,bouhadou2022rbsdes,bouhadou2021reflected,elmansouri2025,grigorova2020optimal,klimsiak2019reflected,marzougue2020predictable,marzougue2023irregular})

In this paper, we consider a complete probability space  $\left(\Omega,\mathcal{F},\mathbb{P}\right)$ and a deterministic finite terminal time $T \in (0,+\infty)$. On this probability space we define a one-dimensional standard Brownian motion $B:=(B_t)_{t \leq T}$ with natural filtration denoted by $\mathbb{G}:=\left(\mathcal{G}_t\right)_{t \leq T}$. Let $\tau: \Omega \to (0,+\infty)$ be a random time not necessary a $\mathbb{G}$-stopping time which models a {\it default time}. The associated indicator process $H:=\mathds{1}_{\llbracket \tau,+\infty \llbracket}$ is termed as the {\it defaultable process}. In credit risk modeling and defaultable markets, the noise and fluctuation is usually created by the Brownian motion $B$ and the defaultable process $H$. Henceforth, at each time $t$ before the expiry time $T$, we have two sources of information: The one that is public to all financial agents and contained in the $\sigma$-algebra $\mathcal{G}_t$, and the other one that is related to the occurrence of the default events. In mathematical finance, in order to solve the problems of pricing and hedging of contingent claims in imperfects  market with default, we proceed using the progressive enlargement of filtration of  $\mathbb{G}$ with respect to the default time $\tau$. Namely, we consider the new flow of information $\mathbb{F}:\left(\mathcal{F}_t\right)_{t \leq T}$ given by $\mathcal{F}_t:=\cap_{\epsilon>0}\mathcal{H}^0_{t+\epsilon}$ with $\mathcal{H}^0_t:=\mathcal{G}_t \vee \sigma\left(\tau \wedge t\right)$ completed by all $\mathbb{P}$-null sets of $\mathcal{F}$. In the new setup $\left(\Omega,\mathcal{F},\mathbb{F},\mathbb{P}\right)$, the defaultable process $H$ becomes an $\mathbb{F}$-submartingale and from the Doob-Meyer decomposition, there is a unique $\mathbb{F}$-predictable process $\Gamma_t$ such that $\Gamma_0=0$ and $M_t=H_t-\Gamma_t$ is a martingale. Moreover, as $\tau$ is regarded as a default time, it is often assumed in the financial studies that $\Gamma$ is absolutely continuous with respect to the Lebesgue measure. Then there exists an $\mathbb{F}$-predictable process $\left(\gamma_t\right)_{t \leq T}$ called the {\it intensity process} of $H$ such that $\Gamma_t=\int_{0}^{t}\left(1-H_s\right)\gamma_s ds=\int_{0}^{t}\gamma_s ds$ as $\Gamma_t=\Gamma_{t \wedge \tau}$, $t \in [0,T]$. So that the explicit decomposition of the {\it compensated defaultable martingale} $(M_t)_{t \leq T}$ is given by:
\begin{equation}
	M_t=H_t-\int_{0}^{t \wedge \tau}\gamma_s ds=H_t-\int_{0}^{t}\gamma_s ds,\quad t \in  [0,T].
	\label{martingale M}
\end{equation}
To avoid the problem of stability of the martingale property in the new setup, we work under the so called $\mathcal{(H)}$-{\it Hypothesis}  or \textit{Immersion property} in the terminology of credit risk modeling, which plays a fundamental and expansive role in the realm of enlarged filtration \cite{bremaud1978changes,kusuoka1999remark,mansuy2006random}. 

In the stochastic basis $\left(\Omega,\mathcal{F},\mathbb{F},\mathbb{P}\right)$ and under the above consideration, we aim to solve the problem of existence and uniqueness of solution for RBSDEs  associated with a terminal condition $\xi$, a driver (or coefficient) $f$ and a reflecting obstacle $\L:=\left(\L_t\right)_{t \leq T}$ (or barrier) which is not necessary right-continuous. More precisely, we force the process $\L$ to have regulated trajectories, meaning that $\L$ is required only to have finite left and right limits. Additionally, we don't require a strong assumptions on the coefficient $f$ as it is assumed to satisfy a kind of stochastic Lipschitz property. Note that, as the obstacle is no longer right-continuous, the state process (first component of the solution) of our RBSDE hesitates this property as well as the reflection process that pushes the solution of the RBSDE to be greater or equal to $\L$ until he reaches the terminal value $\xi$ at the expiry time $T$. 

BSDEs with default time have been an interesting range in the BSDEs theory due to their close link with credit risks, one of the oldest, most basic, and riskiest financial risks, especially default risks which is our context. In this spirit, Dumitrescu et al. \cite{dumitrescu2018bsdes} have considered this type of BSDEs with a generalized driver involving an optional finite variation process in the setup $\left(\Omega,\mathcal{F},\mathbb{F},\mathbb{P}\right)$ taking into account the aforementioned setting. In \cite{dumitrescu2018bsdes}, the authors show the existence and uniqueness of the solution under some suitable square integrability on the data and a Lipschitz condition on the driver based on a martingale representation property with respect to the Brownian motion $B$ and the martingale $M$ defined by \eqref{martingale M}, which we will mention in a few. Additionally, the authors provide a connection with the nonlinear pricing of European contingent claims in a complete imperfect market model with default where the imperfection stems from the non-linearity of the wealth dynamics. More precisely, the authors in \cite{dumitrescu2018bsdes} consider this pricing problem in  a financial market (studied in \cite{TMB})  described by three assets $(S^0_t,S^1_t,S^2_t)_{t \leq T}$, where the evolution of the prices is given by the following system:
\begin{equation}
	\left\{
	\begin{split}
		dS^0_t&=S^0_t r_t dt,\\
		dS^1_t&=S^1_t\left(\mu^1_t dt+\sigma^1_t dB_t\right),\\
		dS^2_t&=S^2_{t-}\left(\mu^2_t dt+\sigma^2_t dB_t-dM_t\right),
	\end{split}
	\right.
	\label{market}
\end{equation} 
where the parameters $\mu^1$, $\mu^2$, $r$, $\sigma^1$ and $\sigma^2$ are $\mathbb{F}$-predictable processes and assumed to be bounded. In the same setting Dumitrescu et al. \cite{dumitrescu2018american} have considered the pricing and hedging problem for an American option with a right continuous and left limited payoff (RCLL, for short) $\L$ in the same financial market model \eqref{market}. The authors in \cite{dumitrescu2018american} introduce the seller's price of the American option and prove that this price coincides with the value function of an optimal stopping problem with $f$-expectation, which in turn corresponds to the solution of a nonlinear RBSDE with RCLL obstacle $\L$ and a Lipschitz driver $f$. In the same way, the authors consider also the buyer's  price of the American option and characterize it via the solution of a RBSDE with an RCLL obstacle $\L$ (see also \cite{GRIGOROVA2021479} for a related study).

This work deals with a more general set of reflected BSDEs with default time $\tau$ and a stochastic Lipschitz driver $f$, where the barrier $\L$ satisfies a weak regularity condition. In the general case of an obstacle with regulated trajectories, we provide a classical description of the first component of the solution as the value function of a classical optimal stopping problem with linear expectation, as well as a generalization of this characterization to the case of non-linear $f$-expectation or $\mathcal{E}^f$-expectation induced by a classical BSDE with default jump and a stochastic Lipschitz driver $f$ in the complete filtered probability space $\left(\Omega,\mathcal{F},\mathbb{F},\mathbb{P}\right)$. More precisely, for a given $[0,T]$-valued $\mathbb{F}$-stopping time $\sigma$, we characterize the state process of our RBSDE as the value function of the following generalized optimal stopping problem:
\begin{equation}
	V(\sigma):=\esssup_{\eta \in \mathcal{T}_{[\sigma,T]}}\mathcal{E}^f_{\sigma,\eta}(\L_{\eta}),
	\label{OS}
\end{equation}
where $\mathcal{T}_{[\sigma,T]}$ denotes the set of stopping times valued a.s. in $[\sigma,T]$ and $\mathcal{E}^f_{\sigma,\eta}(\cdot)$ denotes the $\mathcal{E}^f$-expectation at time $\sigma$ when the terminal time is $\eta$ associated with a stochastic Lipschitz coefficient $f$. Following the work of Grigorova et al. \cite{Grigorova2017} and the terminology of dynamic risk modeling, we interpret $\L$ as a dynamic financial position allowing for a gain $\L_{\eta}$ at time $\eta \in \mathcal{T}_{[0,T]}$. The risk of the position $L_{\eta}$, at time $\sigma$ where $\sigma \in \mathcal{T}_{[0,\eta]}$, is asserted by $-\mathcal{E}^f_{\sigma,\eta}(L_{\eta})$. The objective here is to stop the process $\L$ in such a way that the risk be minimal. In other word, we are dealing the  optimal stopping problem $\Upsilon(\sigma):=-V(\sigma)$ with $\sigma \in \mathcal{T}_{[0,T]}$, where $V(\cdot)$ is defined by \eqref{OS}. Finally, it should be pointed out that RBSDEs with an irregular obstacle and default time do not correspond to a special case of RBSDEs with a L\'{e}vy process or a Poisson random measure, and as we have previously provided, the handling of this kind of BSDE demands certain specific arguments.

The rest of the paper is organized as follows:  In Section \ref{sec2}, we introduce some notation and provide some preliminary needed through the paper. In Section \ref{sec 3} we study the existence and uniqueness result for RBSDE with default time and irregular obstacle under our assumptions supposed for the data. Moreover, we provide a well-known link with a standard optimal stopping problem. Section \ref{sec 4} gives the comparison theorem for solutions to RBSDEs. In section \ref{sec 5}, we provide a link with an optimal stopping problem with $\mathcal{E}^f$-expectations induced by a standard BSDE with default jump and a stochastic Lipschitz driver $f$.
\section{Preliminaries}\label{sec2}
Consider a fixed finite horizon time $T>0$ and let  $(\Omega, \mathcal{F}, \mathbb{P})$ be a complete probability space equipped with two stochastic processes: a one-dimensional standard Brownian motion $B:=\left(B_t\right)_{t \leq T}$ and a jump process $H$ defined by $H_t=\mathbf{1}_{\{\tau \leq t\}}$ for all $t \in[0, T]$, where $\tau$ is a random time which models a {\it default time}. We assume that this default can appear after any fixed time, that is $\mathbb{P}(\tau \geq t)>0$ for all $t \geq 0$ and that $\mathbb{P}\left(\tau \in \left(0,+\infty\right)\right)=1$. We denote by $\mathbb{F}:=\left(\mathcal{F}_t\right)_{t \leq T}$ the augmented filtration generated by $B$ and $H$, which satisfies the usual conditions of right-continuity and completeness. We also assume that $\mathcal{F}_T=\mathcal{F}$. The equality $X=Y$ between any two processes $(X_t)_{t \leq T}$ and $(Y_t)_{t \leq T}$ must be understood in the indistinguishable sense, meaning that  $\mathbb{P}\left(\omega \in \Omega :X_t(\omega)=Y_t(\omega), \forall t \in  [0,T]\right)=1$. The same significance holds for $X \leq Y$. We denote by $\mathcal{P}$ the predictable $\sigma$-algebra on $\Omega \times [0,T]$, and for $x \in \mathbb{R}$, we remember that $x^{+}=\max(x,0)$ and $x^{-}=\max(-x,0)=-\min(x,0)$.\\
We denote by $\mathcal{T}_{\left[ \gamma_1,\gamma_2\right] }$ the set of $[0,T]$-valued $\mathbb{F}$-stopping times $\gamma$ such that $\gamma_1 \leq \gamma \leq \gamma_2$, a.s. for two $[0,T]$-valued $\mathbb{F}$-stopping times $\gamma_1$ and $\gamma_2$ such that $\gamma_1 \leq \gamma_2$ a.s. To simplify the notation, we omit any dependence on $\omega$ of a given process or random function, and by convention, all stochastic integrals are taken null at time zero. Finally and through the paper, we work under the following condition:\\
\textbf{Hypothesis (H):} We suppose that $\left(B_t\right)_{t \leq T}$ is an $\mathbb{F}$-Brownian motion.

Recall that under hypothesis \textbf{(H)}, the martingale representation property holds in the filtration $\mathbb{F}$ with respect to the Brownian motion $B$ and the compensated defaultable martingale $M$ defined by \eqref{martingale M}. Specifically, it can be stated as follows:
\begin{theorem}[\cite{kusuoka1999remark}]
	For every  $\mathbb{F}$-local martingale $N$, there exists a unique pair $\mathbb{R}$-valued $\mathcal{F}$-predictable process $\left(Z_t,U_t\right)_{t \leq T}$ such that 
	$$
	N_t=N_0+\int_{0}^{t} Z_s dB_s+\int_{0}^{t} U_s d M_s,\quad t \in [0,T].
	$$
	Moreover, if $N$ is square integrable, then 
	$$
	\mathbb{E}\int_{0}^{T}\left( \left| Z_s \right|^2 + \left|U_s\right|^2 \gamma_s\right) ds <\infty.
	$$
	\label{Representation property}
\end{theorem}

\begin{remark}
	Note that as the $\mathbb{F}$-predictable compensator of the RCLL process $H$ is continuous, then the defaultable process $H$ is quasi-left continuous, implying in particular, the quasi-left continuity of the filtration due to the martingale representation theorem \ref{Representation property} and Proposition 10.19  in \cite{Kallemberg}.
	\label{quasi left conti}
\end{remark}

We first introduce the notion of processes with regulated trajectories.
\begin{definition}[Regulated processes]
	\begin{itemize}
		\item A process $\mathcal{X}: \Omega \times [0,T] \rightarrow \mathbb{R}$ is said to be a regulated process if for $\mathbb{P}$-almost all $\omega \in \Omega$, the function $t \mapsto \mathcal{X}_t(\omega)$ has finite right limits at each point $t \in [0,T)$, and finite left limits at each point of $(0,T]$.
		\item For any process $\mathcal{X}: \Omega \times [0,T] \rightarrow \mathbb{R}$ with regulated trajectories, we set
		\begin{itemize}
			\item $\mathcal{X}_{s-}=\lim\limits_{u \nearrow s} \mathcal{X}_u$ the left limit of $\mathcal{X}$ at $s \in ]0,T]$ and $\Delta_{-} \mathcal{X}_{s}=\mathcal{X}_{s}-\mathcal{X}_{s-}$ with the convention $\Delta_{-} \mathcal{X}_{0}=0$.
			\item $\mathcal{X}_{s+}=\lim\limits_{s\swarrow u} \mathcal{X}_u$ the right limit of $\mathcal{X}$ at $s \in [0,T[$ and $\Delta_{+} \mathcal{X}_{s}=\mathcal{X}_{s+}-\mathcal{X}_{s}$ with the convention $\Delta_{+} \mathcal{X}_{T}=0$.
		\end{itemize}
		\item Let $\mathcal{K}: \Omega \times [0,T] \rightarrow \mathbb{R}$ be a finite variation, regulated process, then we write $\mathcal{K}=\mathcal{K}^d+\mathcal{K}^c+\mathcal{K}^g$, where the process $\mathcal{K}^c$ is continuous, the RCLL process $\mathcal{K}^d$ equals $\mathcal{K}^d_t=\sum_{0 < s \leq t}\Delta_{-}\mathcal{K}_{s}$
		and the LCRL process $\mathcal{K}^g$ is given by $\mathcal{K}^g_t=\sum_{0 \leq s < t}\Delta_{+}\mathcal{K}_{s}$. This also means that $\mathcal{K}_t=\mathcal{K}^{\ast}_t+\sum_{0 \leq s < t}\Delta_{+}\mathcal{K}_{s}$ where the RCLL process $\mathcal{K}^{\ast}$ represents the right-continuous part of the process $\mathcal{K}$ satisfying
		$\mathcal{K}^{\ast}=\mathcal{K}-\mathcal{K}^g=\mathcal{K}^c+\mathcal{K}^d$ and $\mathcal{K}^g$ it's purely jumping part.
	\end{itemize}
	\label{regulated processes}
\end{definition}
\begin{remark}
	Note that the trajectories of a process with regulated paths have, at most, countably many discontinuities (refer to Corollary II.2.2 in \cite{dudley2011concrete}).
	\label{Many discontinuity}
\end{remark}

Let $\beta>0$ and $(\alpha_t)_{t \leq T}$ be a non-negative $\mathbb{F}$-adapted process $(\alpha_t)_{t\leq T}$. We define an increasing continuous process $A:=(A_t)_{t\leq T}$ defined as $A_t:=\int_0^t \alpha_s^2 ds$. Subsequently, we define the following spaces, assuming $\beta>0$:
\begin{itemize}
	\item $\mathcal{S}^2$: the space of one-dimensional $\mathbb{F}$-optional increasing processes $(K_t)_{t \leq T}$ with regulated trajectories such that
	$$
	\parallel K \parallel_{\mathcal{S}^2}^2 = \mathbb{E}\left[\esssup_{\eta \in \mathcal{T}_{[0,T]}} \left| K_{\eta} \right|^2 \right]<\infty.
	$$
	\item $\mathcal{H}^2$: the space of one-dimensional $\mathcal{P}$-measurable processes $(Z_t)_{t \leq T}$ such that
	$$
	\parallel Z \parallel_{\mathcal{H}^2}^2 = \mathbb{E}\left[ \int_0^{T}  \left| Z_{s} \right|^2 d s\right]<\infty.
	$$
	\item $\mathbb{L}^2_{\beta}\left(\Omega \times [0,T],\mathcal{P},\gamma_t d \mathbb{P}\otimes dt\right)$: The space of $\mathcal{P}$-measurable $\mathbb{R}$-valued process $(U_t)_{t \leq T}$ such that $\mathbb{E}\int_{0}^{T}e^{\beta\mathcal{A}_s} \left| U_s \right|^2 \gamma_s ds<+\infty$.
	
	\item $\mathcal{M}^2_{\gamma,\beta}:=\mathbb{L}^2_{\beta}\left(\Omega \times [0,T],\mathcal{P},\gamma_t d \mathbb{P}\otimes dt\right)$, equipped with the scalar product $\left\langle U,V \right\rangle_{\gamma,\beta}:=\mathbb{E}\int_{0}^{T} e^{\beta \mathcal{A}_s} U_s V_s \gamma_s ds$, for all $(U_{t})_{t \leq T}$, $(V_t)_{t \leq T}$ in $\mathcal{M}^2_{\gamma,\beta}$. For all $U \in \mathcal{M}^2_{\gamma,\beta}$, we have $$\left\|U\right\|^2_{\mathcal{M}^2_{\gamma,\beta}}:=\mathbb{E}\int_{0}^{T}e^{\beta \mathcal{A}_s} \left|U_s\right|^2 \gamma_s ds.$$ 
	We set by convention $\mathcal{M}^2_{\gamma}:=\mathcal{M}^2_{\gamma,0}$.
	\item $\mathbb{L}^2_{\beta}$: the set of one-dimensional $\mathcal{F}_{T}$-measurable random variables $\xi$ such that
	$$\left\|  \xi \right\|^2_{\beta} =\mathbb{E}\left[e^{\beta A_{T}} \left| \xi \right|^2 \right]<\infty.$$
	\item $\mathcal{S}^2_{\beta}$: the space of one-dimensional $\mathbb{F}$-optional processes $(Y_t)_{t \leq T}$ with regulated trajectories such that
	$$
	\parallel Y \parallel_{\mathcal{S}^2_{\beta}}^2 = \mathbb{E}\left[\esssup_{\eta \in \mathcal{T}_{[0,T]}} e^{\beta A_{\eta}} \left| Y_{\eta} \right|^2 \right]<\infty.
	$$
	\item $\mathcal{S}^{2,\alpha}_{\beta}$: the space of one-dimensional $\mathbb{F}$-optional processes $(Y_t)_{t \leq T}$ such that
	$$
	\parallel Y \parallel_{\mathcal{S}^{2,\alpha}_{\beta}}^2 = \mathbb{E}\left[ \int_0^{T} e^{\beta A_{s}} \left| \alpha_s Y_{s} \right|^2 ds \right]<\infty.
	$$
	\item $\mathcal{H}^2_{\beta}$: the space of one-dimensional $\mathcal{P}$-measurable processes $(Z_t)_{t \leq T}$ such that
	$$
	\parallel Z \parallel_{\mathcal{H}^2_{\beta}}^2 = \mathbb{E}\left[ \int_0^{T} e^{\beta A_{s}} \left| Z_{s} \right|^2 ds\right]<\infty.
	$$
	\item $\mathfrak{B}^2_{\beta}:=\left( \mathcal{S}^2_{\beta} \cap \mathcal{S}^{2,\alpha}_{\beta} \right) \times \mathcal{H}^2_{\beta}  \times \mathcal{M}^2_{\gamma,\beta}$ and $\mathfrak{D}^2_{\beta}:=\left( \mathcal{S}^2_{\beta} \cap \mathcal{S}^{2,\alpha}_{\beta} \right) \times \mathcal{H}^2_{\beta} \times \mathcal{S}^2 \times \mathcal{M}^2_{\gamma,\beta}$.
\end{itemize}
\begin{remark}
	Using the definition of the process $\Gamma$ and the fact that $d\Gamma_{s \wedge \tau}=d \Gamma^{\tau}_s=\mathds{1}_{\{s \leq \tau\}} d\Gamma_s$, we have: 
	\begin{equation*}
		\begin{aligned}
			\mathbb{E}\int_{0}^{\infty} D_s d \Gamma^{\tau}_{s  } = \mathbb{E}\int_{0}^{\infty} D_s \mathds{1}_{\{s \leq \tau\}} d\Gamma_s = \mathbb{E}\int_{0}^{\infty} D_s dH^{\tau}_{s}
		\end{aligned}
	\end{equation*}
	for any $\mathcal{P}$-measurable process $(D_t)_{t \geq 0}$ such that $\mathbb{E}\int_{0}^{\infty}\left| D_s\right|dH_s<\infty$. Therefore, $(\Gamma_{s \wedge \tau})_{t \geq 0}$ is the predictable compensator of $(H^{\tau}_{t})_{t \geq 0}=(H_t)_{t \geq 0}$. By uniqueness, we conclude that $\left(\Gamma_t^{\tau}\right)_{t \geq 0}=(\Gamma_t)_{t \geq 0}$ indistinguishably between processes. On the other hand, using the explicit form of the process $\Gamma$ and the positivity of the intensity $\gamma$, we deduce from $\int_{\tau}^{\infty} \gamma_s ds=0$ that $\gamma_t=0$ $\mathbb{P}$-a.s. on the set $\{t >\tau\}$, i.e., $\gamma$ vanishes after the default occurs.
	\label{Rm zero after default}
\end{remark}

\begin{remark}
	Note that as $\gamma$ vanishes after $\tau$, for any $U \in \mathcal{M}^2_{\gamma,\beta}$ we may write $\left\|U\right\|^2_{\mathcal{M}^2_{\gamma,\beta}}:=\mathbb{E}\int_{0}^{T \wedge \tau}e^{\beta \mathcal{A}_s} \left|U_s\right|^2 \gamma_s ds$, moreover, without loss of generality, we may assume that the two processes $\left( U_t\right) _{t \leq T}$ and $\left( U_t\mathds{1}_{\{t\leq\tau\}}\right) _{t \leq T}$ are indistinguishable.
\end{remark}

\section{Problem presentation}
\label{sec 3}
\subsection{RBSDE with one upper reflecting irregular barrier}
We are interested in finding a quadruplet of processes $\left(Y_t,Z_t,K_t,U_t\right)_{t \leq T}$ that satisfies the following reflected BSDE:
\begin{equation}
	\left\{
	\begin{split}
		\text{(i)} &~Y_t= \xi+\int_t^T f(s,Y_s,Z_s,U_s)ds-(K_T-K_t) -\int_t^T Z_s d B_s -\int_t^T  U_s dM_s,\\
		\text{(ii)} &~ Y_t \leq \zeta_t,~ 0 \leq t \leq T,~\text{a.s.,}\\
		\text{(iii)} &~ \text{Skorokhod conditions:}\\
		&~ \text{If } K^{\ast} \text{ denotes the right-continuous part of } K, \text{ then } K^{\ast} \text{ is } \mathcal{P}\text{-measurable and}\\
		&\text{ } 
		\int_0^T (\zeta_{s-}-Y_{s-})dK^{\ast}_s+\sum_{0 \leq s < T}  (\zeta_{s}-Y_s)\Delta_{+}K_s=0.
	\end{split}
	\right.
	\label{basic equation}
\end{equation}
We introduce the definition of the solution for the RBSDE \eqref{basic equation}.
\begin{definition}
	Let $\beta>0$ and $(\alpha_t)_{t \leq T}$ a non negative $\mathbb{F}$-adapted process. A solution to the reflected BSDE associated with terminal variable $\xi$, coefficient $f$ and upper barrier $\zeta$, is a quintuplet of processes $(Y,Z,K,U)$ which satisfy \eqref{basic equation} and belongs to $\mathfrak{D}^2_{\beta}$.
	\label{Definition of the solution}
\end{definition}

The following standard lemma gives an explicit expressions for the jump parts of the reflection processes $K$, and follows from the Skorokhod condition \eqref{basic equation}-(iii). 
\begin{remark}
	The left and right hand jump of the processes $K^{+}$ and $K^{-}$ has the following form: For all $t \in [0,T]$ 
	$$
	\Delta_{-} K_t=\left(Y_t-\zeta_{t-}\right)^{+}\mathds{1}_{\{Y_{t-}=\zeta_{t-}\}\cap\{\Delta_{-} \zeta_t>0\}}, 
	$$
	and
	$$
	\Delta_{+} K_t=\left(Y_{t+}-\zeta_{t}\right)^{+}\mathds{1}_{\{Y_{t}=\zeta_{t}\}\cap\{\Delta_{+} \zeta_t>0\}}.
	$$
	\label{Reflecting property}
\end{remark}

\textbf{The basic assumptions on the data $\left(\xi,f,\zeta\right)$:}
\begin{itemize}
	\item[] \textbf{(H1)} Terminal variable $\xi$:\\
	$\xi$ is an $\mathcal{F}_{T}$-measurable random variable belonging to $\mathbb{L}^2_{\beta}$.
	
	\item[] \textbf{(H2)} Generator $f$:
	\begin{itemize}
		\item[$\bullet$] $\forall y \in \mathbb{R}$, $z \in \mathbb{R}$ and $u \in \mathbb{R}$, the coefficients  $f\left(\cdot,\cdot,y,z,u\right): \Omega \times [0,T] \rightarrow \mathbb{R}$ is $\mathbb{F}$-progressively measurable.
		\item[$\bullet$] {\it Stochastic Lipschitz condition:} There exists there non-negative $\mathbb{F}$-adapted processes $(\mu_t)_{t \leq T}$, $(\theta_t)_{t \leq T}$ and $(\nu_t)_{t \leq T}$ such that
		\begin{itemize}
			\item[(a)] $d\mathbb{P} \otimes dt$-a.s., for
			each $(y_1, z_1, u_1)$, $(y_2, z_2, u_2) \in \mathbb{R}^3$,
			$$
			\left| f(t,y_1,z_1,u_1)-f(t,y_2,z_2,u_2) \right| \leq \mu_t  \left|y_1-y_2 \right|+\theta_t  \left|z_1-z_2 \right|+\nu_t \gamma_t  \left|u_1-u_2 \right|.
			$$
			\item[(b)] There exists $\epsilon >0$ such that $\alpha_t^2:=\mu_t+\theta^2_t +\nu^2_t \gamma_t\geq \epsilon$. 
		\end{itemize}
		\item[$\bullet$] $\frac{f(\cdot,0,0,0)}{\alpha_{\cdot}} \in \mathcal{H}^2_{\beta}$.
	\end{itemize}
	\item[] \textbf{(H3)} Upper obstacle $\zeta$:
	\begin{itemize}
		\item[$\bullet$] The obstacle $\zeta$ is $\mathbb{F}$-optional.
		\item[$\bullet$] The barrier $\zeta$ is a regulated process such that $\xi \leq \zeta_T$ a.s.
		\item[$\bullet$] $\zeta^{-} \in \mathcal{S}^2_{2 \beta}$.
	\end{itemize}
\end{itemize}
\subsection{Existence and uniqueness result}
In this section, we will establish the existence and uniqueness of solutions for the RBSDE \eqref{basic equation} using a {\it modified penalization method}. The reasoning is divided into two main steps:
\begin{enumerate}
	\item We consider the case where the driver $f$ of the RBSDE \eqref{basic equation} does not depend on $\left(y,z\right)$, i.e., $f(\omega,t,y,z)=:g(\omega,t)$ for any $(t,y,z) \in [0,T] \times \mathbb{R}^2$, $\mathbb{P}$-a.s., and $\dfrac{g}{\alpha} \in \mathcal{H}^2_{\beta}$. Then, we  construct a sequence of approximating penalized equations that converges to the solution of the RBSDE \eqref{basic equation} associated with $\left(\xi,g,\zeta\right)$. 
	
	\item Subsequently, we employ a fixed-point argument with an appropriate mapping in a suitable Banach space to establish the result in the general case.
\end{enumerate}
\subsubsection{Existence and uniqueness result for the case when $f$ does not depend on $(y,z,u)$}
In this section, we will prove the existence and uniqueness of a special case of reflected BSDEs with one irregular barrier. More precisely, we are interested in the upper obstacle reflected BSDE with jumps and regulated trajectories (the case where dealing with a lower barrier is quite similar), which takes the form:
\begin{equation}
	\left\{
	\begin{split}
		\text{(i)}~&Y_t=\xi+\int_{t}^{T}g(s)ds-\left(K_T-K_t\right) - \int_{t}^{T}Z_s dB_s-\int_{t}^{T}U_s dM_s,\quad t \in [0,T].\\
		\text{(ii)}~& Y_t \leq \zeta_t,~\forall t \leq T,~\text{a.s.,}\\
		\text{(iii)}~&\text{ Minimality condition: } \int_{0}^{T}\left(\zeta_{s-}-Y_{s-}\right)dK^{\ast}_s+\sum_{s <T}\left(\zeta_s-Y_s\right)\Delta_{+}K_s=0~\text{ a.s.},
	\end{split}
	\right.
	\label{Reflected BSDE with one barrier}
\end{equation}
where $K=K^{\ast}+\Delta_{+} K$, $\dfrac{g}{\alpha} \in \mathcal{H}^2_{\beta}$, $\xi \in \mathbb{L}^2_{\beta}$ and $\zeta^{-} \in \mathcal{S}^2_{2\beta}$.	
\begin{theorem}
	Assume that \textbf{(H1)}, \textbf{(H2)} and \textbf{(H3)} and hold for a sufficient large $\beta$. Then, the RBSDE (\ref{Reflected BSDE with one barrier}) admits a unique solution $\left(Y_t,Z_t,K_t,U_t\right)_{t \leq T} \in \mathfrak{B}^2_{\beta}\times \mathcal{H}^2_{\beta}\times \mathcal{S}^2 \times \mathcal{M}^2_{\gamma,\beta}$.
	\label{existence and uniqueness for the one reflected}
\end{theorem}	
\begin{proof}
	To make the proof as well-constructed and comprehensible as possible, it is done in four steps, each of which includes several findings and notes.
	\paragraph*{Step 1: Construction of the modified penalization schemes.}
	\emph{}
	
	By inspiring on \cite[Section 4]{klimsiak2019reflected}, we consider approximation of the solution to RBSDE associated with parameters $\left(\xi,g,\zeta\right)$. More precisely, for each $n \geq 1$, we consider the following modified penalization version of BSDEs:
	\begin{equation}
		\begin{split}
			Y^n_t=\xi&+\int_{t}^{T}f(s,Y^n_s)ds-\int_{t}^{T} Z^n_s dB_s-\int_{t}^{T} U^n_sdM_s\\
			&-n\int_{t}^{T} \left(Y^n_s-\zeta_s\right)^{+}ds
			-\sum_{t \leq \rho_{n,i} < T} \left(Y^n_{\rho_{n,i}+}-\zeta_{\rho_{n,i}}\right)^{+},\quad t \in [0,T],
		\end{split}
		\label{one reflected: penalization schemes}
	\end{equation}
	where $\left\{\rho_{n,i}\right\}$ is an array of precisely specified stopping times that exhausts $\zeta$'s right-side jumps, constructed as follows: We start with setting 
	\begin{equation*}
		\left\{
		\begin{split}
			& \rho_{1,0}=0,\\
			& \rho_{1,i}=\inf\left\{t > \rho_{1,i-1}:\Delta_{+}\zeta_t >1 \right\} \wedge T,\quad i=1,2,\cdots,k_1,
		\end{split}
		\right.
	\end{equation*}
	and $\rho_{1,k_1+1}=T$ for some $k_1 \geq 1$. Next, for each $n \geq 1$, and for given array $\left\{\rho_{n,i}\right\}$, we set $\rho_{n+1,0}=0$ and
	$$
	\rho_{n+1,i}=\inf\left\{t > \rho_{n+1,i-1}:\Delta_{+}\xi_t >\frac{1}{n+1} \right\} \wedge T,\quad i=1,2,\cdots,j_{n+1},
	$$
	where the index $j_{n+1}$ is chosen such that $\mathbb{P}\left(\rho_{n+1,j_{n+1}} <T\right) \rightarrow 0$ as $n \rightarrow +\infty$ and
	$$
	\rho_{n+1,j_{n+1}+i}=\rho_{n+1,j_{n+1}} \vee \rho_{n,i},\quad i=1,2,\cdots, k_n,~~ \text{ and }~~ k_{n+1}=j_{n+1}+k_n.
	$$
	Finally, we put $\rho_{n+1,k_{n+1}+1}=T$.
	\begin{remark}
		\begin{enumerate}
			\item Note that, since  $\Delta_{+} \zeta_t >\frac{1}{n}$ implies $\Delta_{+} \zeta_t >\frac{1}{n+1}$, then from the above construction it follows that $$ \bigcup_{j=1}^{k_n} \llbracket \rho_{n,j} \rrbracket \subset \bigcup_{j=1}^{k_{n+1}} \llbracket \rho_{n+1,j} \rrbracket,~ \text{ for each } n \geq 1.$$
			So it is natural to include the stopping times from the previous step in the definition of the ones from the current step.
			\item For each $n \geq 1$, and as a result of the construction, the arrays $\left\{\rho_{n,i}\right\}$ are stopping times, satisfying 
			\begin{equation}
				\left[0,T \right]=\left[0,\rho_{n,1} \right] \cup  \bigcup_{i=2}^{k_n+1} \left(   \rho_{n,i-1},\rho_{n,i} \right]. 
				\label{Stopping time definitions}
			\end{equation}
			\item From the BSDE \eqref{one reflected: penalization schemes}, we can express the right-hand jumps of the state process $Y^n$ as follows: $\Delta_{+} Y^n_{t}=\sum_{t= \rho_{n,i}}\left(Y^n_{\rho_{n,i}+}-\zeta_{Y^n_{\rho_{n,i}}}\right)^{+}$.  Following this, we can write $\Delta_{+} Y^n_{\rho_{n,i}}=\left(Y^n_{\rho_{n,i}+}-\zeta_{\rho_{\rho_{n,i}}}\right)^{+}$. Henceforth, we obtain the following equivalent form $Y^n_{\rho_{n,i}}= Y^n_{\rho_{n,i}+}\wedge \zeta_{\rho_{n,i}}$, for each $n \geq 1$ and all $i \in \{1,2,\cdots,k_n\}$.
		\end{enumerate}
		\label{Ghazali}
	\end{remark}
	
	\subparagraph{Methodical Solution Construction for the BSDE \eqref{one reflected: penalization schemes}}
	Observe that from Remark \ref{Ghazali}-(2)-(3), on each interval $(\rho_{n,i-1},\rho_{n,i}]$, where $i=1,2,\cdots,k_n+1$, the BSDE \eqref{one reflected: penalization schemes} becomes a non-reflected BSDEs of the form:
	\begin{equation}
		\begin{split}
			Y^n_t=& Y^n_{\rho_{n,i}+}\wedge \zeta_{\rho_{n,i}}+\int_{t}^{\rho_{n,i}} g(s)ds-n\int_{t}^{\rho_{n,i}} \left(Y^n_s-\zeta_s\right)^{+}ds\\
			&\qquad-\int_{t}^{\rho_{n,i}} Z^n_s dB_s-\int_{t}^{\rho_{n,i}} U^n_s dM_s,\quad t \in (\rho_{n,i-1},\rho_{n,i}],~i \in \{1,2,\cdots,k_n+1\},
		\end{split}
		\label{Penelization equations local}
	\end{equation}	
	with the convention that $Y^n_{0}= Y^n_{0+}\wedge \zeta_{0}$ and $Y^n_T=\xi \wedge \zeta_T=\xi$. Then, to solve the BSDE (\ref{one reflected: penalization schemes}), we consider it on each subinterval: $[0, \rho_{n,1}], (\rho_{n,1}, \rho_{n,2}], \cdots, (\rho_{n,k_n}, T]$, and then construct the corresponding solution backwardly starting from $(\rho_{n,k_n}, T]$.  To this end, we set $\mathfrak{f}_n(s,y):=g(s)-n \left(y-\zeta_s\right)^{+}$. Moreover, we have $\xi \in \mathbb{L}^2_{\beta}$ and
	\begin{equation*}
		\begin{split}
			&\mathbb{E}\int_{0}^{T} e^{\beta A_s} \left| \frac{\mathfrak{f}_n(s,0)}{\alpha_s}\right|^2 ds
			\leq2\left( \mathbb{E}\int_{0}^{T} e^{\beta A_s} \left| \frac{g(s)}{\alpha_s}\right|^2 ds+ \frac{n^2 T}{\epsilon} \mathbb{E}\esssup_{\eta \in \mathcal{T}_{[0,T]} } e^{2 \beta A_{\eta}}\left| \zeta_{\eta}^{-}\right|^2    \right) .
		\end{split}
	\end{equation*}
	Hence, from Theorem \ref{Existence and uniqueness unconstrained BSDE}, the BSDE \eqref{Penelization equations local} associated with $(\xi,\mathfrak{f}_n)$ has a unique solution $\left(Y^n,Z^n,U^n\right)\in \mathfrak{B}^2_{\beta} \times \mathcal{H}^2_{\beta} \times \mathcal{M}^2_{\gamma,\beta}$. Following an inductive process, we construct a unique solution $\left(Y^n,Z^n,U^n\right)\in \mathfrak{B}^2_{\beta} \times \mathcal{H}^2_{\beta} \times \mathcal{M}^2_{\gamma,\beta}$ to the BSDE \eqref{one reflected: penalization schemes} for each $n \geq 1$. Additionally, the solution may be stated in the following shorter form:
	\begin{equation}
		\begin{split}
			Y^n_t=\xi&+\int_{t}^{T} g(s)ds-\int_{t}^{T}dK^{n}_s-\int_{t}^{T} Z^n_s dB_s-\int_{t}^{T}U^n_s dM_s,\quad t \in [0,T],
		\end{split}
		\label{Short dynamic of Yn}
	\end{equation}
	where $\left(K^{n}_t\right)_{t \leq T}$ is a regulated process with decomposition given as follows:
	\begin{equation}
		K^{n}_t:=K^{n,\ast}_t+\sum_{0 \leq s < t} \Delta_{+} K^{n}_s:=n\int_{t}^{T} \left(Y^n_s-\zeta_s\right)^{+}ds+\sum_{0 \leq \rho_{n,i} < t} \left(Y^n_{\rho_{n,i}+}-\zeta_{\rho_{n,i}} \right)^{+}.
		\label{K-n}
	\end{equation}
	\paragraph{Part 2: Uniform estimation for the sequence $\left\{\left( Y^n,Z^n,K^{+,n},K^{-,n},N^n\right) \right\}_{n \geq 1}$}
	\begin{lemma}
		There exists a positive constant $\mathfrak{C}_{\beta}$ independent of $n$ such that for all $\beta >1$
		\begin{equation*}
			\begin{split}
				&\left\|Y^n \right\|^2_{\mathfrak{B}^2_{\beta}}+\left\|Z^n \right\|^2_{\mathcal{H}^2_{\beta}}+\left\|U^n \right\|^2_{\mathcal{M}^2_{\gamma,\beta}}+\mathbb{E}\left|K^{n}_T \right|^2 \\
				& \leq \mathfrak{c}_{\beta} \left( \mathbb{E}e^{\beta A_T}\left|\xi \right|^2+  \mathbb{E} \esssup_{\eta \in \mathcal{T}_{[0,T]} } e^{2 \beta A_{\eta}} \left| \xi_{\eta}^{+}\right|^2  +  \mathbb{E} \int_{0}^{T} e^{\beta A_s} \left| \dfrac{g(s)}{\alpha_s} \right|^2 ds \right) .
			\end{split}
		\end{equation*}
		\label{lemma of unifrom estimation}
	\end{lemma}
	
	\begin{proof}
		By applying Corollary \ref{Application of Ito formula} to the dynamics of the process $Y^n$ given by equation \eqref{Short dynamic of Yn} with right-continuous part
		$
		Y^{n,\ast}_t=Y_0-\int_{0}^{t}g(s)ds+K^{n,\ast}_t+\int_{0}^{t} Z^n_s dB_s+\int_{0}^{t} U^n_s dM_s
		$ and $\Delta_{+} Y^n = \Delta_{+} K^{n}$, we can derive 
		\begin{equation}
			\begin{split}
				&e^{\beta A_t} \left| Y^n_t \right|^2+\beta \int_{t}^{T} e^{\beta A_s} \left| Y^n_s \right|^2 dA_s+\int_{t}^{T} e^{\beta A_s} \left| Z^n_s \right|^2 ds \\
				&   =e^{\beta A_T} \left| \xi \right|^2+2 \int_{t}^{T} e^{\beta A_s} Y_{s}g(s)ds-2\int_{t}^{T} e^{\beta A_s}Y^n_{s-}dK^{n,\ast}_s\\
				& \qquad -2 \int_{t}^{T} e^{\beta A_s} Y^n_{s}Z^n_sdB_s-2\int_{t}^{T} e^{\beta A_s}Y^n_{s-} U^n_s dM_s -\sum_{t <s \leq T}e^{\beta A_s} \left| \Delta_{-} Y^n_s\right|^2\\
				&\qquad -\sum_{t \leq s < T} e^{\beta A_s} \left|\Delta_{+} Y^n_s \right|^2-2\sum_{t \leq s < T} e^{\beta A_s} Y^n_s \Delta_{+}K^{n}_s.
			\end{split}
			\label{basic Itos formula}
		\end{equation}
		As $H$ is an RCLL quasi-left continuous process (as its has a continuous compensator) and $K^{n}$ are left-continuous adapted processes, then we have $\Delta_{-} K^{n} \Delta_{-} H=0$. Hence
		\begin{equation}
			\sum_{t <s \leq T}e^{\beta A_s} \left| \Delta_{-} Y^n_s\right|^2=\sum_{t <s \leq T}e^{\beta A_s} \left|  \Delta_{-} K^{n}_s\right|^2+\sum_{t <s \leq T}e^{\beta A_s} \left| U^n_s\right|^2  \left| \Delta M_s\right|^2.
			\label{left jump}
		\end{equation}
		On the other hand, we have
		\begin{equation}
			\sum_{t <s \leq T}e^{\beta A_s} \left| U^n_s\right|^2  \left| \Delta M_s\right|^2=\int_{t}^{T}e^{\beta A_s} \left| U^n_s\right|^2  dH_s=\int_{t}^{T}e^{\beta A_s} \left| U^n_s\right|^2 d\left[M\right]_s,
			\label{integral M}
		\end{equation}
		Next, using Holder's inequality and the relation  $2 \sqrt{ab} \leq \epsilon a + \frac{1}{\epsilon} b$ for every $a,b \geq 0$, $\epsilon > 0$, we may drive the following inequality for each $\beta >1$
		\begin{equation}
			\begin{split}
				2 \int_{t}^{T} e^{\beta A_s} Y_{s}g(s)ds& \leq 2 \left(\int_{t}^{T} e^{\beta A_s} \left| Y^n_{s} \right|^2 dA_s\right)^{\frac{1}{2}} \left(\int_{t}^{T} e^{\beta A_s} \left| \dfrac{g(s)}{\alpha_s} \right|^2 ds\right)^{\frac{1}{2}} \\
				& \leq \left(\beta-1\right) \int_{t}^{T} e^{\beta A_s} \left| Y^n_{s} \right|^2 dA_s+\dfrac{1}{\beta-1} \int_{t}^{T} e^{\beta A_s} \left| \dfrac{g(s)}{\alpha_s} \right|^2 ds
			\end{split}
			\label{estimation of Yn and g}
		\end{equation}
		By plugging \eqref{left jump}, \eqref{integral M}, and \eqref{estimation of Yn and g} into \eqref{basic Itos formula}, we get
		\begin{equation}
			\begin{split}
				&e^{\beta A_t} \left| Y^n_t \right|^2+ \int_{t}^{T} e^{\beta A_s} \left| Y^n_s \right|^2 dA_s+\int_{t}^{T} e^{\beta A_s} \left| Z^n_s \right|^2 d s+ \int_{t}^{T}e^{\beta A_s} \left| U^n_s\right|^2 d\left[M\right]_s \\
				&  \leq  e^{\beta A_{T}} \left| \xi\right|^2  +\dfrac{1}{\beta-1} \int_{t}^{T} e^{\beta A_s} \left| \dfrac{g(s)}{\alpha_s} \right|^2 ds-2\int_{t}^{T} e^{\beta A_s}Y^n_{s}dK^{n}_s\\
				&\quad -2 \int_{t}^{T} e^{\beta A_s} Y^n_{s}Z^n_s dB_s-2\int_{t}^{T} e^{\beta A_s}Y^n_{s-}U^n_s dM_s
			\end{split}
			\label{both sides expectation}
		\end{equation}
		
		Now, considering the constructed arrays of stopping times $\left\{\rho_{n,i}\right\}$ using \eqref{Stopping time definitions} and the expression \eqref{K-n}, we obtain
		\begin{equation*}
			\begin{split}
				\int_{0}^{T}e^{\beta A_s}Y^n_{s}dK^{n}_s
				=\sum_{i=1}^{k_n+1} \int_{\rho_{n,i-1}}^{\rho_{n,i}} e^{\beta A_s}Y^n_{s}dK^{n}_s
				\geq -\int_{0}^{T}e^{\beta A_s}\zeta^{-}_s dK^{n}_s
			\end{split}
		\end{equation*}
		
			%
		Using this together with the inequality $2 ab \leq \epsilon a^2 +\frac{1}{\epsilon} b^2$, $\forall \epsilon>0$, we get
		\begin{equation}
			-2\int_{t}^{T} e^{\beta A_s}Y^n_{s}dK^{n}_s \leq \epsilon \esssup_{\eta \in \mathcal{T}_{[0,T]}} e^{2 \beta A_{\eta}}\left| \zeta_{\eta}^{-} \right|^2+\frac{1}{\epsilon}\left| K^n_T \right|^2.  
			\label{Square down}
		\end{equation}
		
		Next, we need to state the following proposition:
		\begin{proposition}
			The stochastic integral $\mathcal{W}:=\big(\int_{0}^{t}e^{\beta A_s} \big\{Y^n_{s} Z^n_sdB_s+Y^n_{s-}U^n_sdM_s\big\} \big)_{t \leq T}$ is a uniformly integrable $\mathbb{F}$-martingale with zero expectation.
			\label{Martingale Proposition}
		\end{proposition}
		\begin{proof}
			First, using the left-continuity of the process $\left(e^{\beta A_t}Y_{t-}\right)_{t \in [0,T]}$, we get
			$$
			\sup_{s \in [0,T]}  e^{\beta A_s}\left|Y^n_{s-} \right|^2 = \sup_{s \in \mathbb{Q} \cap [0,T]}  e^{\beta A_s}\left|Y^n_{s-} \right|^2 \leq \esssup_{\eta \in \mathcal{T}_{[0,T]}} e^{\beta A_{\eta}}\left|Y^n_{\eta} \right|^2 \qquad \text{a.s.}
			$$
			Next, notice that, since $(M_t)_{t \leq T}$ is a finite variation martingale, and using the definition of the defaultable process $(H_t)_{t \leq T}$, we get $d\left[M,M\right]_s=d\left[M\right]_s=\left(\Delta H_s\right)^2=\Delta H_s=dH_s$, hence: For all $\nu \in \mathcal{T}_{[0,T]}$, we have
			\begin{equation*}
				\begin{split}
					&\mathbb{E}\left[\sqrt{\int_{0}^{\nu} e^{2 \beta A_s} \left| Y^{n}_{s-} \right|^2 \left| U^n_s\right|^2 d \left[M \right]_s}\right]\\ &\leq \mathbb{E}\left[\sqrt{\esssup_{\eta \in \mathcal{T}_{[0,T]}} e^{\beta A_{\eta}}\left|Y^n_{\eta} \right|^2 \int_{0}^{\nu} e^{ \beta A_s} \left| U^n_s\right|^2 d \left[M \right]_s}\right]\\
					&\leq\dfrac{1}{2} \mathbb{E}\left[\esssup_{\tau \in \mathcal{T}_{[0,T]}} e^{\beta A_{\tau}}\left|Y^n_{\tau} \right|^2 \right]
					+\dfrac{1}{2}\mathbb{E}\left[\int_{0}^{T} e^{ \beta A_s} \left| U^n_s\right|^2 \gamma_s ds \right]=\dfrac{1}{2} \left\| Y^n \right\|^2_{\mathcal{S}^2_{\beta}} +\dfrac{1}{2} \left\| U^n \right\|^2_{\mathcal{M}^2_{\gamma,\beta}}.
				\end{split}
			\end{equation*}
			Similarly, we can obtain,
			\begin{equation*}
				\begin{split}
					\mathbb{E}\left[ \sqrt{\int_{0}^{\nu} e^{2 \beta A_s} \left| Y^{n}_{s} \right|^2\left| Z^n_s\right|^2  ds}\right]  \leq \dfrac{1}{2} \left\| Y^n \right\|^2_{\mathcal{S}^2_{\beta}} +\dfrac{1}{2} \left\| Z^n \right\|^2_{\mathcal{H}^2_{\beta}}.
				\end{split}
			\end{equation*}
			Then, the claim is obtained using the right-continuity of the processes $\mathcal{W}$, the Burkholder-Davis-Gundy's inequality, and Theorem I.51 on page 38 in \cite{Protter}.
		\end{proof}
	
	After returning to \eqref{both sides expectation}, taking expectation on it's both sides at $t = 0$, using \eqref{Square down} and the result of  Proposition \ref{Martingale Proposition}, we obtain
	\begin{equation}
		\begin{split}
			&\mathbb{E} \int_{0}^{T} e^{\beta A_s} \left| Y^n_s \right|^2 dA_s+\mathbb{E} \int_{0}^{T}e^{\beta A_s} \left| Z^n_s \right|^2 d s+\mathbb{E} \int_{0}^{T} e^{\beta A_s} \left| U^n_s\right|^2  \gamma_s ds  \\
			&  \leq  \mathbb{E}  e^{\beta A_{T}} \left| \xi\right|^2  +\dfrac{1}{\beta-1} \mathbb{E} \int_{0}^{T} e^{\beta A_s} \left| \dfrac{g(s)}{\alpha_s} \right|^2 ds
			+\epsilon\mathbb{E} \esssup_{\eta \in \mathcal{T}_{[0,T]}} e^{2 \beta A_{\eta}}\left| \zeta_{\eta}^{-} \right|^2+\dfrac{1}{\epsilon}\mathbb{E}\left| K^n_T \right|^2.
		\end{split}
		\label{The view of this}
	\end{equation}
	Writing equation \eqref{Short dynamic of Yn} forwardly, then squaring, using Holder's inequality and taking the expectation, we obtain
	\begin{equation}
		\begin{split}
			\mathbb{E}\left| K^n_T \right|^2 & \leq 5 \left(\left|Y^n_0 \right|^2+\mathbb{E}e^{\beta A_T}\left|\xi \right|^2+\dfrac{1}{\beta}  \mathbb{E} \int_{0}^{T} e^{\beta A_s} \left| \dfrac{g(s)}{\alpha_s} \right|^2 ds\right.\\
			&\qquad\qquad\left.+\mathbb{E} \int_{0}^{T}e^{\beta A_s} \left| Z^n_s \right|^2 d s+\mathbb{E} \int_{0}^{T} e^{\beta A_s} \left| U^n_s\right|^2  \gamma_s ds\right)
		\end{split}
		\label{Kn est}
	\end{equation}
	Following this, choosing $\epsilon>5$ and plugging this into \eqref{The view of this}, we derive the existence of a constant $\mathfrak{c}_{\beta}>0$ that depends only on $\beta$ such that 
	\begin{equation*}
		\begin{split}
			&\mathbb{E} \int_{0}^{T} e^{\beta A_s} \left| Y^n_s \right|^2 dA_s+\mathbb{E} \int_{0}^{T}e^{\beta A_s} \left| Z^n_s \right|^2 d s+\mathbb{E} \int_{0}^{T} e^{\beta A_s} \left| U^n_s\right|^2  \gamma_s ds  \\
			&  \leq \mathfrak{c}_{\beta}\left(  \mathbb{E}  e^{\beta A_{T}} \left| \xi\right|^2  + \mathbb{E} \int_{0}^{T} e^{\beta A_s} \left| \dfrac{g(s)}{\alpha_s} \right|^2 ds
			+\mathbb{E} \esssup_{\eta \in \mathcal{T}_{[0,T]}} e^{2 \beta A_{\eta}}\left| \zeta_{\eta}^{-} \right|^2\right). 
		\end{split}
	\end{equation*}
	Therefore, from this and estimation \eqref{Kn est}, we derive  
	\begin{equation}
		\begin{split}
			&\mathbb{E} \int_{0}^{T} e^{\beta A_s} \left| Y^n_s \right|^2 dA_s+\mathbb{E} \int_{0}^{T}e^{\beta A_s} \left| Z^n_s \right|^2 d s+\mathbb{E} \int_{0}^{T} e^{\beta A_s} \left| U^n_s\right|^2  \gamma_s ds+\mathbb{E}\left| K^n_T \right|^2  \\
			&  \leq \mathfrak{c}_{\beta}\left(  \mathbb{E}  e^{\beta A_{T}} \left| \xi\right|^2  + \mathbb{E} \int_{0}^{T} e^{\beta A_s} \left| \dfrac{g(s)}{\alpha_s} \right|^2 ds
			+\mathbb{E} \esssup_{\eta \in \mathcal{T}_{[0,T]}} e^{2 \beta A_{\eta}}\left| \zeta_{\eta}^{-} \right|^2\right). 
		\end{split}
		\label{First unif est}
	\end{equation}
	
	It remains to prove the uniform estimation for the sequence of random variables $\left\{\esssup_{\eta \in \mathcal{T}_{[0,T]} }e^{\beta A_{\eta}}\left|Y^n_{\eta} \right|^2 \right\}_{n \geq 1}$. To do so, we back to \eqref{both sides expectation}, then using \eqref{Square down}, and a similar computations as the one used in the proof of Proposition \ref{Martingale Proposition} along with \eqref{First unif est}, we obtain
	\begin{equation*}
		\begin{split}
			\mathbb{E}\esssup_{\eta \in \mathcal{T}_{[0,T]}} e^{\beta A_{\eta}} \left| Y^n_{\eta} \right|^2
			& \leq  \mathfrak{c}_{\beta}\left(  \mathbb{E}  e^{\beta A_{T}} \left| \xi\right|^2  + \mathbb{E} \int_{0}^{T} e^{\beta A_s} \left| \dfrac{g(s)}{\alpha_s} \right|^2 ds
			+\mathbb{E} \esssup_{\eta \in \mathcal{T}_{[0,T]}} e^{2 \beta A_{\eta}}\left| \zeta_{\eta}^{-} \right|^2\right). 
		\end{split}
	\end{equation*}
	Then the proof of Lemma \ref{lemma of unifrom estimation} is complete.
\end{proof}
\paragraph{Step 3: Convergence of the sequence $\left\{Y^n,Z^n,U^{n}\right\}_{n \geq 1}$ in $\mathfrak{B}^2_{\beta} \times \mathcal{H}^2_{\beta} \times \mathcal{M}^2_{\gamma,\beta}$ to the limiting process $\left(Y,Z,U\right)$.}
\begin{itemize}
	\item \textit{Stage 1: There exists an $\mathbb{F}$-optional process $Y:=(Y_t)_{t \leq T}$ with regulated trajectories such that $Y\in \mathcal{S}^2_{\beta}$ and $Y^n \searrow Y$ on $[0,T]$.}\\
	From \eqref{Penelization equations local} and since $\zeta_{\rho_{n,k_n+1}} \wedge Y^n_{\rho_{n,k_n+1}+}=\zeta_T\wedge Y^n_{T}=\zeta_T\wedge \zeta_T=\zeta_T$ and $\mathfrak{f}_{n+1}(s,y) \leq \mathfrak{f}_{n}(s,y)$ for any $s \in [0,T]$, a.s for all $y \in \mathbb{R}$ and due to the fact that $Y^n$ has RCLL paths on each subinterval $(\rho_{n,k-1},\rho_{n,k}]$, we can apply Theorem \ref{Comparison theorem} and Remark \ref{Comparison BSDE} starting from $(\rho_{n,k_n},T]$ to deduce that $Y^{n}_t \geq Y^{n+1}_t$ for all $t \in (\rho_{n,k_n},T]$. Similarly, using the same comparison principle, we find that $\zeta_{\rho_{n,k}} \wedge Y^{n}_{\rho_{n,k}+} \geq \zeta_{\rho_{n,k}} \wedge Y^{n+1}_{\rho_{n,k}+}$, and consequently, $Y^{n}_t \geq Y^{n+1}_t$ for all $t \in (\rho_{n,k_n-1},\rho_{n,k_n}]$. By repeating this process on each subinterval, the fact that $Y^{n}_{0}=\zeta_0 \vee Y^{n}_{0+} \geq \zeta_0 \vee Y^{n}_{0+}=Y^{n+1}_0$ by convention, and considering the construction of the solution $(Y^n_t)_{t \leq T}$ presented in the first step of the current proof, we can conclude that $Y^n_t \geq Y^{n+1}_t$ for all $t \in [0,T]$. Additionally, note that, this comparison result can also be obtained using BSDE \eqref{one reflected: penalization schemes}, Remark \ref{Ghazali}-(1) and  Proposition 3.3 in \cite{klimsiak2019reflected}. Hence, there exists an $\mathbb{F}$-optional process $(Y_t)_{t \leq T}$ such that $Y^n_t \searrow Y_t$, $\forall t \in [0,T]$ a.s. Moreover,  thanks to the monotonic limit theorem for regulated processes (see for instance \cite[Theorem 2.10]{klimsiak2019reflected}) the limit process Y has regulated trajectories. Now, by employing Fatou's lemma and the uniform estimate presented in Lemma \ref{lemma of unifrom estimation}, we get
	\begin{equation*}
		\begin{split}
			&\mathbb{E}\esssup_{\eta \in \mathcal{T}_{[0,T]}} e^{\beta A_{\eta}} \left| Y_{\eta}\right|^2 \\
			&\leq \liminf_{n\rightarrow \infty} \mathbb{E}\esssup_{\eta \in \mathcal{T}_{[0,T]}} e^{\eta A_{\eta}} \left| Y^{n}_{\eta}\right|^2   \\
			& \leq \mathfrak{c}_{\beta} \left( \mathbb{E} e^{\beta A_T}\left|\zeta_T \right|^2  + \mathbb{E} \int_{0}^{T} e^{\beta A_s} \left| \dfrac{g(s)}{\alpha_s} \right|^2 ds+  \mathbb{E}  \esssup_{\eta \in \mathcal{T}_{[0,T]} } e^{2 \beta A_{\eta}} \left| \zeta_{\eta}^{-}\right|^2  \right).
		\end{split}
	\end{equation*}
	\item \textit{Stage 2: Cauchy property of the sequence $\{Z^n,U^n\}_{n \geq 1}$.}
	
	For each $n > p \geq 0$, from \eqref{one reflected: penalization schemes}, we have the following dynamics for the process $Y^n-Y^p$,
	\begin{equation}
		d\left(Y^n_s-Y^p_s\right)=d\left(K^{n}_s-K^{p}_s\right)+\left(Z^n_s-Z^p_s\right)dB_s+\left(U^n_s-U^p_s\right) dM_s,~ Y^n_T-Y^p_T=0.
		\label{dynamics of Yn-Yp}
	\end{equation}
	Then, Theorem \ref{Ito's formula Theorem} and the expression \eqref{integral M} implies that
	\begin{equation}
		\begin{split}
			&\left| Y^n_t-Y^p_t \right|^2+\int_{t}^{T}  \left| Z^n_s-Z^p_s \right|^2 ds+\int_{t}^{T}e^{\beta A_s}  \left|  U^n_s-U^p_s\right|^2 d\left[M\right]_s  \\
			&   =-2\int_{t}^{T} \left( Y^n_{s-}-Y^p_{s-} \right)  d\left(  K^{n,\ast}_s- K^{p,\ast}_s\right)-2 \int_{t}^{T}  \left( Y^n_{s}-Y^p_{s}\right) \left( Z^n_s- Z^p_s\right) dB_s \\
			& \quad -2\int_{t}^{T} \left( Y^n_{s-}-Y^p_{s-}\right) \left( U^n_s-U^p_s\right)dM_s -\sum_{t \leq s < T}  \left|\Delta_{+} \left( Y^n_s-Y^p_s\right)  \right|^2 \\
			&\quad-2\sum_{t \leq s < T}\left(  Y^n_s-Y^p_s \right)  \Delta_{+}\left( K^{n}_s-K^{p}_s\right).
		\end{split}
		\label{basic Itos formula for the convergence of the triplet}
	\end{equation}
	By considering the fact that the process $\left(\Delta_{+} K_t\right)_{t \in [0,T]}$ is $\mathbb{F}$-predictable, and the jumping times of the process $\left( \Delta_{-}\left( Y^n_t-Y^p_t\right)\right)_{t \in [0,T]}$ are  totally inaccessible stopping times, we can write $\left( Y^n_s-Y^p_s\right)\Delta_{+}K^{q}_s=\left( Y^n_{s-}-Y^p_{s-}\right)\Delta_{+}K^{q}_s$, for all $s \in [0,T]$, and each $q \in  \{n,p\}$. After that, using this formula into equation \eqref{basic Itos formula for the convergence of the triplet}, we get
	\begin{equation}
		\begin{split}
			&\left| Y^n_t-Y^p_t \right|^2+\int_{t}^{T}  \left| Z^n_s-Z^p_s \right|^2 ds+\int_{t}^{T}   \left| U^n_s-U^p_s\right|^2 d \left[M \right]_s   \\
			&  \leq -2\int_{t}^{T} \left( Y^n_{s-}-Y^p_{s-} \right)  d\left(  K^{n}_s- K^{p}_s\right)-2 \int_{t}^{T}  \left( Y^n_{s}-Y^p_{s}\right) \left( Z^n_s- Z^p_s\right) dB_s \\
			& \quad -2\int_{t}^{T} \left( Y^n_{s-}-Y^p_{s-}\right)\left(U^n_s-U^p_s\right) dM_s
		\end{split}
		\label{Need some clarifications}
	\end{equation}
	Now, to control Stieltjes integrals appearing on the right-hand side of \eqref{Need some clarifications}, we state following auxiliary result:
	\begin{lemma}
		For each $n \geq 1$, and for any $t \in [0,T]$, we have
		$$
		\int_{t}^{T} \left(Y^n_{s-}-\zeta_{s-}\right)dK^{n}_s=\int_{t}^{T}\left(Y^n_s-\zeta_s\right)dK^{n,\ast}_s+\sum_{t \leq s < T}\left(Y^n_s-\zeta_s\right)\Delta_{+}K^{n}_s \geq 0,~\text{a.s.}
		$$
		\label{result needed for majoration}
	\end{lemma}
	\begin{proof}
		From the definition of the process $\left(\Delta_{+} K^n_t\right)_{t \in [0,T]}$, we have
		\begin{equation}
			\sum_{t \leq s < T}\left(Y^n_s-\zeta
			_s\right)\Delta_{+}K^{n}_s=\sum_{t \leq \sigma_{n,i} < T} \left(Y^n_{\rho_{n,i}}-\zeta_{\rho_{n,i}}\right)\left(Y^n_{\rho_{n,i}+}-\zeta_{\rho_{n,i}} \right)^{+}.
			\label{in virtue of this equation}
		\end{equation}
		Now, let $n \in \mathbb{N}$ be fixed and assume that there exists an index $i \in \{1,2,\cdots,k_n\}$ such that $t \leq \sigma_{n,i} < T$ and $\left(Y^n_{\rho_{n,i}}-\zeta_{\rho_{n,i}}\right)\left(Y^n_{\rho_{n,i}+}-\zeta_{\rho_{n,i}} \right)^{+}<0$. Therefore, we necessarily have $Y_{\rho_{n,i}}<\zeta_{\rho_{n,i}}$ and $Y^n_{\sigma_{n,i}+}>\zeta_{\rho_{n,i}}$. Moreover, from \eqref{one reflected: penalization schemes}, we have  $\Delta_{+}Y^n_{\rho_{n,i}}=\Delta_{+} K^{n}_{\rho_{n,i}}$. Hence, $Y^n_{\rho_{n,i}}=\zeta_{\rho_{n,i}}$, which leads to a contradiction. Consequently, for every $i \in \{1,2,\cdots,k_n\}$, we have $\big(\zeta_{\rho_{n,i}}-Y_{\rho_{n,i}}\big)\left(Y^n_{\rho_{n,i}+}-\xi_{\rho_{n,i}} \right)^{-}\geq 0$. This inequality, in particular, yields  $\sum_{t \leq s < T}\left(Y^n_s-\zeta_s\right)\Delta_{+}K^{n}_s \geq 0$ by virtue of equality (\ref{in virtue of this equation}). On the other hand, we have
		$$
		\int_{t}^{T}\left(Y^n_{s}-\zeta_s\right)dK^{n,\ast}_s=n\int_{t}^{T}\left(Y^n_{s}-\xi_s\right)\left(Y^n_{s}-\zeta_s\right)^{+}ds=n\int_{t}^{T}\left(  \left(Y^n_{s}-\zeta_s\right)^{+}\right)^2 ds \geq 0.
		$$
		Hence, we deduce that $\int_{t}^{T} \left(Y^n_{s-}-\zeta_{s-}\right)dK^{n}_s \geq 0$, for every $n \geq 1$.
	\end{proof}

	On the other hand, using the left continuity of the process  $\left( \zeta_{t-}-Y^j_{t-}\right)_{t \leq T}$ for $j \in \{n,p\}$, along with the inequalities
	$$
	\sup_{s \in [0,T]}  \left( Y^n_{s-}-\zeta_{s-} \right)^{+}=\sup_{s \in [0,T] \cap \mathbb{Q}}  \left( Y^n_{s-}-\zeta_{s-} \right)^{+}   \leq \esssup_{\eta \in \mathcal{T}_{[0,T]}} \left( Y^n_{\eta}-\xi_{\eta} \right)^{+},$$
	together wit the result of Lemmas \ref{lemma of unifrom estimation}, \ref{result needed for majoration}, and the Cauchy-Schwarz inequality, it follows that
	\begin{equation}
		\begin{split}
			&\mathbb{E}\int_{0}^{T}  \left| Z^n_s-Z^p_s \right|^2 d s+\mathbb{E}\int_{0}^{T}    \left| U^n_s-U^p_s\right|^2 \gamma_s ds  \\
			&  \leq \mathfrak{c}_{\beta}\left\lbrace \left( \mathbb{E}  \esssup_{\eta \in \mathcal{T}_{[0,T]}} \left| \left( Y^n_{\eta}-\zeta_{\eta} \right)^{+}\right|^2  \right)^{\frac{1}{2}}  +\left( \mathbb{E} \esssup_{\eta \in \mathcal{T}_{[0,T]}} \left|\left( Y^p_{\eta}-\zeta_{\eta} \right)^{+}\right|^2\right)^{\frac{1}{2}}  \right\rbrace .
		\end{split}
		\label{passing to the limit in n,p to infinity}
	\end{equation}
	
	Now, we have to demonstrate the subsequent achievement:
	\begin{lemma}
		$$
		\lim\limits_{n \to +\infty}\mathbb{E}   \esssup_{\eta \in \mathcal{T}_{[0,T]}} \left| \left( Y^n_{\eta}-\zeta_{\eta} \right)^{+}\right|^2   =0.
		$$
		\label{Lemma barriers convergences}
	\end{lemma}
\begin{proof}
	let the triplet $\left(\hat{Y}^n,\hat{Z}^n,\hat{U}^n\right)$ be the solution of the following classical BSDE:
	$$
	\hat{Y}^n_t=\xi +\int_{t}^{T}g(s)ds-n\int_{t}^{T} \left(\hat{Y}^n_s-\zeta_s\right)ds-\int_{t}^{T} \hat{Z}^n_s dM_s-\int_{t}^{T}\hat{U}^n_s dM_s,\quad t \in [0,T].
	$$
	Since $-\left(y-\zeta_s\right)=-\left(y-\zeta_s\right)^{+}+\left(y-\zeta_s\right)^{-}\geq- \left(y-\zeta_s\right)^{+}$. Then, from the comparison Theorem \ref{Comparison for BSDEs}-(i), we deduce that $\hat{Y}^n_t \geq Y^n_t$ for all $t \leq T$. Next, for any $\eta \in \mathcal{T}_{[0,T]}$, an integration by part formula leads to
	\begin{equation}
		\hat{Y}^n_{\eta}=\mathbb{E}\left[e^{-n(T-\eta)}\xi+\int_{\eta}^{T}e^{-n(s-\eta)}g(s)ds+n\int_{\eta}^{T}e^{-n(s-\eta)}\zeta_s ds \mid \mathcal{F}_{\eta}\right].
		\label{representation of the process hat Y}
	\end{equation}
	It is clear that,
	$$
	e^{-n(T-\eta)}\xi+n\int_{\nu}^{T}e^{-n(s-\eta)}\xi_s ds \xrightarrow[n \rightarrow +\infty]{}\xi\mathds{1}_{\{\eta=T\}}+\xi_{\eta}\mathds{1}_{\{\eta<T\}},\quad\mathbb{P}\text{-a.s.} \text{ and in } \mathbb{L}^2.
	$$
	In addition, by Holder's inequality, we obtain
	\begin{equation}
		\left|\int_{\eta}^{T}e^{-n(s-\eta)} g(s) ds  \right|^2 \leq  \left(\int_{\eta}^{T} e^{\beta A_s}\left|\dfrac{g(s)}{\alpha_s}\right|^2 ds \right)\left(\int_{\eta}^{T} e^{-2n(s-\eta)-\beta A_s} dA_s\right).
		\label{Majoration with Mahjoub}
	\end{equation}
	Thus $
	\mathbb{E} \left[ \int_{\eta}^{T} e^{-n(s-\eta)} g(s)ds \mid \mathcal{F}_{\eta} \right] \xrightarrow[n \rightarrow \infty]{} 0$, $\mathbb{P}$\text{-a.s.} Now, we define
	$$
	\hat{y}^n_t:=e^{-n(T-t)}\xi+\int_{t}^{T} e^{-n(s-t)} g(s)ds+n\int_{t}^{T}e^{-n(s-t)}\zeta_sds,\quad 0 \leq t \leq  T.
	$$
	From (\ref{representation of the process hat Y}), the definition above and the cross-section theorem, it's clear that $\hat{Y}^n_t-\zeta_t=\mathbb{E}\left(\hat{y}^n_t-\zeta_t \mid \mathcal{F}_t\right)$, $\forall t \in [0,T]$. Following this and using Jensen's inequality, Doob's maximal quadratic inequality (see  \cite[Theorem 1.43, Page 11]{jacod2013limit}), Theorem I.9 in \cite{Protter}, and Remark A.1 in \cite{Grigorova2017}, we have
	\begin{equation}
		\begin{split}
			\mathbb{E}\esssup_{\eta \in \mathcal{T}_{[0,T]}}\left| \left(\hat{Y}^n_{\eta}-\zeta_{\eta}\right)^{+}\right|^2 
			&\leq \mathbb{E}\esssup_{\eta \in \mathcal{T}_{[0,T]}} \left|\mathbb{E}\left[\left(  \hat{y}^n_{\eta}-\zeta_{\eta}\right)^{+}\mid \mathcal{F}_{\eta}\right]\right|^2   \\
			&\leq \mathbb{E}\esssup_{\eta \in \mathcal{T}_{[0,T]}} \left|\mathbb{E}\left[\esssup_{\eta \in \mathcal{T}_{[0,T]}}\left(  \hat{y}^n_{\eta}-\zeta_{\eta}\right)^{+} \mid \mathcal{F}_{\eta}\right]\right|^2   \\
			& =\mathbb{E}\sup_{t \in [0,T]} \left|\mathbb{E}\left[\esssup_{\eta \in \mathcal{T}_{[0,T]}}\left(  \hat{y}^n_{\eta}-\zeta_{\eta}\right)^{+}\mid \mathcal{F}_{t}\right]\right|^2   \\
			&\leq 4\mathbb{E}\esssup_{\eta \in \mathcal{T}_{[0,T]}}\left|\left(  \hat{y}^n_{\eta}-\zeta_{\eta}\right)^{+}\right|^2 ,
		\end{split}
		\label{plugging this to get 0 in xi}
	\end{equation}
	Meanwhile, it is not hard to show that the sequence $\left\lbrace X^n\right\rbrace_{n \geq  1}$ defines as
	\begin{equation*}
		X^n_t:=e^{-n(T-t)}\xi +n\int_{t}^{T}e^{-n(s-t)}\zeta_s ds-\zeta_t,\quad 0 \leq t \leq T,
	\end{equation*}
	is uniformly convergent in $t$ to $0$, $\mathbb{P}$-a.s. In particular, this convergence holds also for  $\left\lbrace (X^{n})^+\right\rbrace_{n \geq  1}$. Next, using \eqref{plugging this to get 0 in xi}, the fact that $\hat{y}^n_t-\zeta_t=X^n_t+\int_{t}^{T} e^{-n(s-t)} g(s)ds$, the basic inequality $(a+b)^{+} \leq a^{+}+\left| b\right|$, $\forall a,b \in \mathbb{R}$ , and the Lebesgue dominated convergence theorem, indicates that
	\begin{equation*}
		\begin{split}
			&\lim\limits_{n \rightarrow +\infty} \mathbb{E}\esssup_{\eta \in \mathcal{T}_{[0,T]}} \left|\left(\hat{Y}^n_{\eta}-\zeta_{\eta} \right)^{+} \right|^2\\
			&\leq 4\lim\limits_{n \rightarrow +\infty} \mathbb{E}\esssup_{\eta \in \mathcal{T}_{[0,T]}} \left|\left(\hat{y}^n_{\eta}-\zeta_{\eta} \right)^{+} \right|^2\\
			& \leq 8 \left( \lim\limits_{n \rightarrow +\infty} \mathbb{E}\esssup_{\eta \in \mathcal{T}_{[0,T]}} \left|\left(X^n_{\eta} \right)^{+} \right|^2+\lim\limits_{n \rightarrow +\infty} \mathbb{E}\sup_{t \in [0,T]} \left| \int_{t}^{T}e^{-n(s-t)} g(s) ds\right|^2 \right) \xrightarrow[n \rightarrow +\infty]{} 0.
		\end{split}
	\end{equation*}
	Since $\left( Y^n_t-\zeta_t \right)^{+} \leq \left( \hat{Y}^n_t-\zeta_t\right)^{+}$, $\forall t \in [0,T]$, we deduce that
	\begin{equation*}
		\lim\limits_{n \rightarrow +\infty} \mathbb{E}\esssup_{\eta \in \mathcal{T}_{[0,T]}} \left|\left(Y^n_{\eta}-\zeta_{\eta} \right)^{+} \right|^2=0.
	\end{equation*}
	This concludes the proof of Lemma \ref{Lemma barriers convergences}.
\end{proof}

By going back to \eqref{passing to the limit in n,p to infinity}, applying the result of Lemma \ref{Lemma barriers convergences}, and subsequently taking the limit, we may now establish
\begin{equation}
	\lim\limits_{n,p\rightarrow +\infty}\left(\left\| Z^n-Z^p \right\|^2_{\mathcal{H}^2}+\left\| U^n-U^p \right\|^2_{\mathcal{M}^2_{\gamma}} \right)=0,
	\label{Zn,Un CV}
\end{equation}
which implies that $(Z^n,U^n)_{n \geq 1}$ is a Cauchy sequence in the Banach space $\mathcal{H}^2 \times \mathcal{M}^2_{\gamma}$. Thus, there exists a pair of processes $(Z,U) \in \mathcal{H}^2 \times \mathcal{M}^2_{\gamma}$ such that
\begin{equation}
	\lim\limits_{n\rightarrow +\infty}\left(\left\| Z^n-Z \right\|^2_{\mathcal{H}^2}+\left\| U^n-U \right\|^2_{\mathcal{M}^2_{\gamma}} \right)=0.
	\label{CV ZU}
\end{equation}

On the other hand, by using (\ref{basic Itos formula for the convergence of the triplet}) along with the same computing techniques employed in the proof of Lemma \ref{lemma of unifrom estimation}, the results from Lemma \ref{Lemma barriers convergences}, \ref{lemma of unifrom estimation}, and the convergence (\ref{Zn,Un CV}), we can conclude that
$$
\lim\limits_{n,p \to +\infty}\mathbb{E} \esssup_{\eta \in \mathcal{T}_{[0,T]}} \left| Y^n_{\eta}-Y^p_{\eta} \right|^2 =0.
$$
From  Proposition 2.1 in \cite{Grigorova2017}, we know that $\mathcal{S}^2$ is a Banach space.  Henceforth, there exists a unique process $Y \in \mathcal{S}^2$ such that 
\begin{equation}
	\lim\limits_{n \to +\infty}\mathbb{E} \esssup_{\eta \in \mathcal{T}_{[0,T]}} \left| Y^n_{\eta}-Y_{\eta} \right|^2 =0.
	\label{CV Y}
\end{equation}
Now, a classical application of Fatou's lemma, the Lebesgue dominated convergence theorem, along with the uniform estimation provided by Lemma \ref{lemma of unifrom estimation} and the convergence results (\ref{CV ZU}) and (\ref{CV Y}), allows us to obtain the integrability condition satisfied by the limiting process $\left(Y_t,Z_t,N_t\right)_{t \leq T}$ as given in the following Lemma:
\begin{lemma}
	The limited process $\left(Y_t,Z_t,N_t\right)_{t \leq T}$ denied by (\ref{CV ZU}) and (\ref{CV Y}) verifies
	$$\left\| Y\right\|^2_{\mathcal{S}^2_{\beta}}+\left\| Y\right\|^2_{\mathcal{S}^{2,\alpha}_{\beta}}+\left\|Z\right\|^2_{\mathcal{H}^2_{\beta}}+\left\|U\right\|^2_{\mathcal{M}^2_{\gamma,\beta}}\leq \mathfrak{c}_{\beta},$$
	where the constant $\mathfrak{c}_{\beta}$ is determined by the right-hand side of the uniform estimation provided by Lemma \ref{lemma of unifrom estimation}.
	\label{Unifor estima and Fatous Lemma}
\end{lemma}

\item \textit{Stage 3: Convergence of the sequence $\{K^n\}_{n \geq 1}$.}\\
Coming back to the backward equation \eqref{Short dynamic of Yn}, and using the results from the previous steps related to the convergence of the sequence $\left\{Y^n,Z^n,U^n\right\}_{n \geq 1}$, we obtain
$$
\lim\limits_{n,p\to+\infty}\mathbb{E} \esssup_{\eta \in \mathcal{T}{[0,T]}} \left| K^n_{\eta}-K^p_{\eta} \right|^2  =0.
$$
Therefore, there exists an $\mathbb{F}$-optional process $K \in \mathcal{S}^2$ such that $K^n \rightarrow K$ as $n \rightarrow +\infty$ in the $\mathcal{S}^2$ space. Moreover, it is worth noting that $K$ has non-decreasing paths. On the other hand, utilizing the uniform estimation satisfied by $\left\{K^{n}\right\}_{n \geq 1}$, we can deduce, by employing the Lebesgue Dominated Convergence theorem, that $\mathbb{E}\left| K_T\right|^2=\lim\limits_{n \rightarrow+\infty} \mathbb{E}\left| K^{n}_T\right|^2 \leq \mathfrak{c}_{\beta}$, where the constant $\mathfrak{c}_{\beta}$ is defined by the right-hand side of the inequality given in Lemma \ref{lemma of unifrom estimation}. Henceforth, the process $K$ has finite left and right limits on $[0,T]$. In other word, $K$ is non-decreasing process with regulated trajectories.

\item \textit{Stage 4: The limiting process $(Y,Z,K,U)$ verifies of the BSDE \eqref{Reflected BSDE with one barrier}-(i).}\\
Passing to the limit term by term in $ \mathbb{L}^2(\Omega,d\mathbb{P})$ as $n \to +\infty$ in \eqref{Short dynamic of Yn}, we obtain
\begin{equation}
	\begin{split}
		Y_t=\xi&+\int_{t}^{T} g(s)ds-\left(K_T-K_t\right)-\int_{t}^{T} Z_s dB_s-\int_{t}^{T}U_s dM_s,\quad t \in [0,T],
	\end{split}
	\label{Classico GHANA}
\end{equation}

\end{itemize}

\paragraph*{Step 4: Skorokhod condition}
Given that the sequence $\{Y^n,K^n\}_{n \geq 1}$ converges towards $(Y,K)$ with respect to the $\left\| \cdot \right\|_{\mathcal{S}^2}$ norm, it also converges uniformly in $t$ in probability. In particular, we get that the measure $dK^{n,\ast}$ tends to $dK^{\ast}$ and that $\Delta_{+}K^n$ tends to $\Delta_{+} K$, where we have used the path-wise decomposition of the non-decreasing regulated process $K$ given by $K=K^{\ast}+\sum_{0 \leq s <\cdot}\Delta_{+} K_s$, with $K^{\ast}=K^c+K^d$ presented in the Definition \ref{regulated processes}. By considering that $\Delta_{+}K^{n}_t=\sum_{\rho_{n,i}=t} \left(Y^n_{\rho_{n,i}+}-\zeta_{\rho_{n,i}}\right)^{+}$ along with the definition of the arrays $\{\rho_{n,i}\}_{n \in \mathbb{N}}$ and letting $n$ tends to $+\infty$, we arrive at the expression $
\Delta_{+}K_t=\left(Y_{t+}-\zeta_t\right)^{+} \mathds{1}_{\{Y_t=\zeta_t\}\cap\{\Delta_{+}\zeta_{t}>0\}}$. Moreover, as the measure $dK^{n,\ast}$ converges weakly to $dK^{\ast}$ in probability, we obtain
$$
\int_{0}^{T}(\zeta_{s-}-Y^n_{s-})dK^{n,\ast}_s \xrightarrow[n \rightarrow+\infty]{\mathbb{P}}\int_{0}^{T}(\zeta_{s-}-Y_{s-})dK^{c,\ast}_s.
$$
Next, from the expression of $K^{n,\ast}$, we deduce that $\int_{0}^{T}(\zeta_{s-}-Y^n_{s-})dK^{n,\ast}_s \leq 0$, for all $n \geq 1$, which implies that $\int_{0}^{T}(\zeta_{s-}-Y_{s-})dK^{\ast}_s \leq 0$. Meanwhile, using Lemma \ref{Lemma barriers convergences}, we deduce that $Y_{\eta} \leq \zeta_{\eta}$ for all $\eta \in \mathcal{T}_{[0,T]}$. By applying the section theorem, we conclude that $Y_t \leq \zeta_t$ for all $t \leq T$ a.s. Consequently, we have $Y_{t-} \leq \zeta_{t-}$ and then $\int_{0}^{T}(\zeta_{s-}-Y_{s-})dK^{c,\ast}_s \geq 0$. Hence, $\int_{0}^{T}(\zeta_{s-}-Y_{s-})dK^{\ast}_s = 0$.\\ Finally, by applying once more Theorem \ref{Ito's formula Theorem} to the dynamic \eqref{Classico GHANA}, and using the Burkholder-Davis-Gundy's inequality, one can derive that for any $\beta > 0$,
\begin{equation*}
\begin{split}
	&\mathbb{E}\esssup_{\eta \in \mathcal{T}_{[0,T]} } e^{ \beta A_{\eta}} \left| Y_{\eta}\right|^2 + \mathbb{E}\int_{0}^{T}e^{\beta A_s} \left( \left| Y_s \right|^2 dA_s+\left\lbrace \left| Z_s \right|^2 + \left| U_s\right|^2 \gamma_s \right\rbrace ds \right)   +\mathbb{E}\left|K_T \right|^2 \\
	& \leq \mathfrak{c}_{\beta} \left( \mathbb{E}e^{\beta A_T}\left|\xi \right|^2 + \mathbb{E} \int_{0}^{T} e^{\beta A_s} \left| \dfrac{g(s)}{\alpha_s} \right|^2 ds+\mathbb{E} \esssup_{\eta \in \mathcal{T}_{[0,T]} } e^{2 \beta A_{\eta}} \left| \zeta_{\eta}^{-}\right|^2  \right).
\end{split}
\end{equation*}
This concludes the proof oh Theorem \ref{existence and uniqueness for the one reflected}.
\end{proof}

\subsubsection{Existence and uniqueness result for RBSDE with general coefficient}
The generator $f$ is now considered to be in a general form, implying that it can depend on the parameters $(y, z,u)$.\\
The theorem below presents the main result of the paper.
\begin{theorem}
	Assume that \textbf{(H1)}, \textbf{(H2)}, and \textbf{(H3)} hold for a sufficiently large $\beta>0$. Then, the RBSDE \eqref{basic equation} admits a unique solution $\left(Y_t,Z_t,K_t,U_t\right)_{t \leq T} \in \mathfrak{B}^2_{\beta}\times \mathcal{H}^2_{\beta}\times \mathcal{S}^2 \times \mathcal{M}^2_{\gamma,\beta}$.
	\label{basic theorem}
\end{theorem}	

\begin{proof}
	The desired result will be achieved by finding a fixed point of the contraction of the function $\Psi$, which is defined as follows:
	
	Let $\mathfrak{D}^2_{\beta}:=\mathcal{S}^{2,\alpha}_{\beta} \times \mathcal{H}^2_{\beta}\times \mathcal{M}^2_{\gamma,\beta}$, endowed with the norm
	$$
	\left\|\left(Y,Z,U\right) \right\|_{\beta}=\left(\mathbb{E}\left[\int_{0}^{T}e^{\beta A_s}\left(\left| \alpha_s Y_s\right|^2+\left| Z_s \right|^2 +\left| U_s \right|^2 \gamma_s  \right)ds \right]\right)^{\frac{1}{2}}.
	$$
	Let $\Psi$ be the map from $\mathfrak{D}^2_{\beta}$ into itself which associates $\left(y,z,u\right)$ to  $\left(Y,Z,U\right)$ through $\Psi$, where the process $\big(Y,Z,K,U\big)$ is the  solution of the DRBSDE \eqref{Reflected BSDE with one barrier} associated with data  $\big(\xi,f\left(t,u_t,z_t,u_t\right),\zeta\big)$.
	
	Note that, from assumption \textbf{(H2)}, we have
	\begin{equation*}
		\begin{split}
			\left| \dfrac{f(s,y_s,z_s,u_s)}{\alpha_s}\right|^2 &\leq 4 \dfrac{\mu^2_s \left| y_s\right|^2+\theta^2_s \left| z_s\right|^2 +\nu_s^2 \gamma_s^2\left|u_s \right|^2   }{\alpha^2_s}+4\left| \dfrac{f(s,0,0,0)}{\alpha_s}\right|^2.
		\end{split}
	\end{equation*}  
	Hence, since $\mu_s^2 \leq \alpha^4_s$, $\theta^2_s \leq \alpha_s^2$, and $\nu_s^2 \gamma_s^2 \leq \gamma_s \alpha^2_s$, we get
	\begin{equation*}
		\begin{split}
			&\mathbb{E}\int_0^T e^{\beta A_s} \left| \dfrac{f(s,y_s,z_s,u_s)}{\alpha_s}\right|^2 ds \\
			& \leq 4 \left( \mathbb{E}  \int_0^T e^{\beta A_s}\left| y_s\right|^2 d A_s+ \mathbb{E}  \int_0^T e^{\beta A_s}\left( \left|z_s\right|^2 +\left| u_s \right|^2 \gamma_s \right) ds \right. \\	&\qquad\qquad\left.+\mathbb{E}\int_0^T e^{\beta A_s} \left| \dfrac{f(s,0,0,0)}{\alpha_s}\right|^2 ds\right)<\infty.
		\end{split}
	\end{equation*}
	Then, by Theorem \ref{existence and uniqueness for the one reflected}, the introduced mapping $\Psi$ is well defined.
	
	Now, let
	$\left(y^{\prime},z^{\prime},u^{\prime}\right)$ be another triple of $\mathfrak{D}^2_{\beta}$ and $\left(Y^{\prime},Z^{\prime},U^{\prime}\right)=\Psi\left(y^{\prime},z^{\prime},u^{\prime}\right)$. Set $\bar{\mathfrak{S}}:=\mathfrak{S}-\mathfrak{S}^{\prime}$, for $\mathfrak{S}:=Y,Z,K$ and $U$.
	
	By applying Corollary \ref{Application of Ito formula}, and performing some standard computations, we obtain
	\begin{equation}
		\begin{split}
			&\beta \mathbb{E} \int_{0}^{T} e^{\beta A_s} \left| \bar{Y}_s \right|^2 dA_s+\mathbb{E} \int_{0}^{T}e^{\beta A_s} \left| \bar{Z}_s \right|^2 d s+\mathbb{E} \int_{0}^{T} e^{\beta A_s} \left|\bar{U}_s \right|^2  \gamma_s ds  \\
			&  \leq 2 \mathbb{E}\int_{0}^{T}\bar{Y}_s \left(f(s,y_s,z_s,u_s)-f(s,y^{\prime}_s,z^{\prime}_s,u^{\prime}_s)\right) ds \\ &\qquad+2\mathbb{E} \int_{0}^{T} e^{\beta A_s}\bar{Y}_{s-}d \bar{K}^{\ast}_s+2\mathbb{E} \sum_{0 \leq s < T} e^{\beta A_s}\bar{Y}_{s}\Delta_{+} \bar{K}_s
		\end{split}
		\label{Itos formula for the general case}
	\end{equation}
	Thanks to the Skorokhod conditions on $K$, we have
	\begin{equation}
		\begin{split}
			\int_{0}^{T} e^{\beta A_s}\bar{Y}_{s-}d \bar{K}^{\ast}_s=\int_{0}^{T} e^{\beta A_s}\bar{Y}_{s-}\left( d K^{\ast}_s-d K^{\prime,\ast}_s\right) \leq 0,
		\end{split}
		\label{Skorokhod condition for the right part}
	\end{equation}
	and
	\begin{equation}
		\begin{split}
			\sum_{0 \leq s < T} e^{\beta A_s}\bar{Y}_{s}\Delta_{+} \bar{K}_s
			=\sum_{0 \leq s < T} e^{\beta A_s}\left(Y_s-\zeta_s \right)   \Delta_{+} K^{\prime,-}_s-\sum_{0 \leq s < T} e^{\beta A_s}\left(\zeta_s-Y^{\prime}_s \right) \Delta_{+} K^{-}_s 
			\leq 0.
		\end{split}
		\label{Skorokhod condition for the right limit}
	\end{equation}
	On the other hand, by utilizing the stochastic Lipschitz condition satisfied by the coefficient $f$ and the basic inequality $2ab \leq \frac{1}{\epsilon}a^2+\epsilon b^2$ for all $\epsilon > 0$, we can establish the following inequality for any $\beta >1$,
	\begin{equation}
		\begin{split}
			&2\bar{Y}_s \left(f(s,y_s,z_s,u_s)-f(s,y^{\prime}_s,z^{\prime}_s,u^{\prime}_s)\right) ds  \\
			&\leq 2\left| \bar{Y}_s\right|  \left( \mu_s \left|\bar{y}_s \right| +\theta_s \left|\bar{z}_s \right| +\nu_s \gamma_s \left| \bar{u}_s\right|   \right) ds\\
			&\leq \left(\beta-1\right)\left| \bar{Y}_s \alpha_s\right|^2  ds+\dfrac{1}{\beta-1}\left(   \left| \bar{y}_s \alpha_s\right|^2+   \left| \bar{z}_s\right|^2 +\left|\bar{u}_s \right|^2 \gamma_s \right)ds.
		\end{split}
		\label{General case}
	\end{equation}
	
	By substituting inequalities \eqref{Skorokhod condition for the right part}, \eqref{Skorokhod condition for the right limit}, and \eqref{General case} into \eqref{Itos formula for the general case}, we obtain, for $\beta >2$,
	\begin{equation*}
		\begin{split}
			& \mathbb{E} \int_{0}^{T} e^{\beta A_s} \left| \bar{Y}_s \right|^2 dA_s+\mathbb{E} \int_{0}^{T}e^{\beta A_s} \left| \bar{Z}_s \right|^2 ds+\mathbb{E} \int_{0}^{T} e^{\beta A_s} \left|\bar{U}_s \right|^2  \gamma_s ds  \\
			&  \leq \mathfrak{c}_{\beta}  \left( \mathbb{E} \int_{0}^{T} e^{\beta A_s} \left| \bar{y}_s \right|^2 dA_s+\mathbb{E} \int_{0}^{T}e^{\beta A_s} \left| \bar{z}_s \right|^2 d s+\mathbb{E} \int_{0}^{T} e^{\beta A_s} \left|\bar{u}_s \right|^2  \gamma_s ds \right),
		\end{split}
	\end{equation*}
	where $\mathfrak{c}_{\beta} \in ]0,1[$.
	
	Then, The mapping $\Psi$ is a contraction and thus possesses a unique fixed point $(Y, Z)$, which indeed belongs to $\mathcal{S}^{2,\alpha}_{\beta} \times \mathcal{H}^2_{\beta}$. Moreover, there exists $(N, K^{+}, K^{-}) \in \mathcal{M}^2_{\beta} \times \mathcal{S}^2 \times \mathcal{S}^2$ with $K^{\pm}_0 = 0$ such that $(Y, Z, K^{+}, K^{-}, N)$ constitutes a unique solution to the reflected BSDE \eqref{basic equation} associated with $\left(\xi,f,\zeta\right)$.
\end{proof}

\subsection{Reflected BSDEs with one lower irregular barrier and a standard optimal stopping problem}
\subsubsection{Existence and uniqueness result}
First, we replace condition \textbf{(H3)} related to the upper obstacle to the following one, 
\subparagraph{\textbf{(H3')} Lower obstacle $\L$}
\begin{itemize}
	\item The obstacle $\L:=(\L_t)_{t \leq T}$ is an $\mathbb{F}$-optional process.
	\item The barrier $\L$ is a regulated process in the sense of Definition \ref{regulated processes} such that $\xi \geq \L_T$ a.s.
	\item $\L^{+} \in \mathcal{S}^2_{2 \beta}$.
\end{itemize}

Next, we consider the following BSDE with one lower irregular reflecting barrier $\L$:
\begin{equation}
	\left\{
	\begin{split}
		\text{(i)} &~Y_t= \xi+\int_t^T f(s,Y_s,Z_s,U_s)ds+(K_T-K_t) -\int_t^T Z_s d B_s -\int_t^T  U_s dM_s,\\
		\text{(ii)} &~ Y_t \geq \L_t,~ 0 \leq t \leq T,~\text{a.s.,}\\
		\text{(iii)} &~ \text{Skorokhod conditions:}\\
		&~ \text{ If } K^{\pm,\ast} \text{ denotes the right-continuous part of } K^{\pm} \text{ then } K^{\pm,\ast} \text{ is predictable and }\\
		&\text{ } 
		\int_0^T (Y_{s-}-\L_{s-})dK^{\ast}_s+\sum_{0 \leq s < T}  (Y_s-\L_t)\Delta_{+}K_s=0.
	\end{split}
	\right.
	\label{basic equation lower}
\end{equation}
The solution of the RBSDE \eqref{basic equation lower} is given similarly as in Definition \ref{Definition of the solution} for the case of BSDE with one upper irregular reflecting barrier.

Following this and using an analogous argument as the one used in the proof of Theorems \ref{existence and uniqueness for the one reflected} and \ref{basic theorem}, we may show the following result.\\
\begin{theorem}
	Assume that the triplet $(\xi,f,\L)$ satisfies \textbf{(H1)}, \textbf{(H2)} and \textbf{(H3')} for a sufficiently large $\beta>0$. Then, the RBSDE \eqref{basic equation lower} admits a unique solution $\left(Y_t,Z_t,K_t,U_t\right)_{t \leq T} \in \mathfrak{B}^2_{\beta}\times \mathcal{H}^2_{\beta}\times \mathcal{S}^2 \times \mathcal{M}^2_{\gamma,\beta}$.
	\label{Lower existence Thm}
\end{theorem}	

\begin{remark}
	Note that the notion of a solution of a BSDE with one upper or lower
	reflecting barrier is closely linked. Namely, A quadruplet $(Y,Z,K,U)$ is a solution for the BSDE with a upper reflecting irregular barrier $\zeta$, a coefficient $f$ and a terminal value $\xi$ if and only if $(-Y,-Z,K,-U)$ is a solution for the BSDE with a reflecting lower irregular barrier associated with $(-\xi,-f,-\zeta)$. 
\end{remark}
\begin{remark}
	We point out that a quadruplet $(Y,Z,K,U)$ is a solution of the BSDE \eqref{basic equation lower}-(i) if and only if $Y_{\eta}=\xi+\int_{\eta}^{T}f(s,Y_s,Z_s,U_s)ds+(K_T-K_\eta) -\int_\eta^T Z_s d B_s -\int_\eta^T  U_s dM_s$, a.s. for all $\eta \in \mathcal{T}_{[0,T]}$ (refer to \cite[Theorem IV.84]{dellacherie1975probabilites}).
	\label{Remark of the starting stp}
\end{remark}
We additionally provide an integrable property fulfilled by the components of the state process $Y$ for the RBSDE \eqref{basic equation lower} given by the decomposition \eqref{form of the process for Ito}, namely, the RCLL semimartingales $Y^{\ast}$ and the  purely discontinuous part $\sum_{0 \leq s < \cdot} \Delta_{+}Y_s$.
\begin{remark}
	Note that due to the non-decreasingness property of $K^{\ast}$ and $\sum_{0 \leq s <\cdot} \Delta_{+} K_s$, and the fact that $\Delta_{+}Y_s=-\Delta_{+} K_s$, $\forall s \in [0,T]$ a.s., we have $\mathbb{E}\left( \sum_{0 \leq s < T} \Delta_{+} Y_s \right)^2=\mathbb{E}\left( \sum_{0 \leq s < T} \Delta_{+} K_s \right)^2 \leq \mathbb{E} \left|  K_T \right|^2$. Next, from the definition of the solution $(Y_t,Z_t,K_t,U_t)_{t \leq T}$ (see Definition \ref{Definition of the solution}), we have $ \mathbb{E} \left|  K_T \right|^2<\infty$, then $\mathbb{E}\left( \sum_{0 \leq s < T} \Delta_{+} Y_s \right)^2<+\infty$. Returning to the corresponding decomposition \eqref{form of the process for Ito} of $Y$, and using the fact that $Y \in \mathcal{S}^2_{\beta}$ together with the right-continuity of $Y^{\ast}$ and Remark A.1 in \cite{Grigorova2017}, we get
	$$
	\mathbb{E}\sup_{0 \leq t \leq T}\left| Y^{\ast}_t \right|^2 \leq 2\left(\mathbb{E}\esssup_{\eta \in \mathcal{T}{[0,T]}}e^{\beta A_{\eta}}\left|Y_{\eta} \right|^2+\mathbb{E}\left( \sum_{0 \leq s < T} \Delta_{+} Y_s \right)^2 \right)<+\infty.
	$$ 
	\label{Remark on the right-continuous part}
\end{remark}
\subsubsection{Links with a standard optimal stopping problem}
We now give a well-known characterization of the first component of the solution for the RBSDE \eqref{basic equation lower} as the value of a given optimal stopping problem.\\
Given a data $(\xi,f,\L)$, satisfying conditions \textbf{(H1)}, \textbf{(H2)} and \textbf{(H3')}. Denotes by $\left(Y_t,Z_t,K_t,U_t\right)_{t \leq t}$ the unique solution of the RBSDE \eqref{basic equation lower}. Consider an optimal stopping problem with gain process given at each time $\eta \in \mathcal{T}_{[0,T]}$ by
$$
\mathsf{G}_\eta=\int_{0}^{\eta}f(s,Y_s,Z_s,U_s)ds+\L_{\eta} \mathds{1}_{\{\eta<T\}}+\xi \mathds{1}_{\{\eta=T\}}.
$$
Then, we have the following proposition
\begin{proposition}
	Let $(Y_t,Z_t,K_t,U_t)_{t \leq T}$ be the unique solution of the RBSDE \eqref{basic theorem} associated with $(\xi,g,\L)$. Then
	$$
	Y_t=\esssup_{\eta \in \mathcal{T}_{[t,T]}} \mathbb{E}\left[ \int_{t}^{\eta}f(s,Y_s,Z_s,U_s)ds+\L_{\eta} \mathds{1}_{\{\eta<T\}}+\xi \mathds{1}_{\{\eta=T\}} \mid \mathcal{F}_t \right],\quad t \in [0,T]. 
	$$
	In other word, $Y_0$ is the value of the optimal stopping problem with payoff given by the process $(\mathsf{G}_t)_{t \leq T}$.
	\label{Characetrization of lower obstacle}
\end{proposition}

\begin{proof}
	Set 
	$$
	S_t:=\esssup_{\eta \in \mathcal{T}_{[t,T]}} \mathbb{E}\left[ \int_{t}^{\eta}f(s,Y_s,Z_s,U_s)ds+\L_{\eta} \mathds{1}_{\{\eta<T\}}+\xi \mathds{1}_{\{\eta=T\}} \mid \mathcal{F}_t \right],\quad t \in [0,T].
	$$
	Then $\mathsf{S}:=\big(S_t +\int_{0}^{t}f(s,Y_s,Z_s,U_s)ds\big)_{t \leq T}$ is the Snell envelope of the process $\left(\mathsf{G}_t\right)_{t \leq T}$. As $Y \in \mathfrak{B}^2_{\beta}$, $Z \in \mathcal{H}^2_{\beta}$, $U \in \mathcal{M}^2_{\beta,\gamma}$, $\L^{+} \in \mathcal{S}^2_{2\beta}$, $\xi^{-} \leq \L^{-}_T$, $\xi \in \mathbb{L}^2_{\beta}$ and $\frac{f(\cdot,0,0,0)}{\alpha} \in \mathcal{H}^2_{\beta}$, we deduce  from Proposition 3.10 in \cite{klimsiak2019reflected}, that $\mathsf{S}$ is an $\mathbb{F}$-supermartingale of class (D). Hence, based on the results of \cite{Aspects} using Mertens decomposition, we derive the existence of an increasing
	process $\mathcal{K}$ with regulated trajectories that satisfies $\mathcal{K}_0 =0$ and a  martingale $\mathcal{M}$ with  the representation $\mathcal{M}_t=\mathcal{M}_0+\int_{0}^{t}\mathcal{Z}_s dB_s+\int_{0}^{t} \mathcal{U}_s dM_s$ (see Theorem \ref{Representation property}), such that 
	$$
	\mathsf{S}_t=\mathsf{S}_T+\int_{t}^{T} d\mathcal{K}_s-\int_{t}^{T}\mathcal{Z}_s dB_s-\int_{t}^{T} \mathcal{U}_s dM_s,\quad t \in [0,T].
	$$
	Then, as $\mathsf{S}_T=\xi+\int_{0}^{T}f(s,Y_s,Z_s,U_s)ds$, the quadruplet $(S,\mathcal{Z},\mathcal{K},\mathcal{U})$ satisfies the following BSDE
	$$
	S_t=\xi+\int_{t}^{T}f(s,Y_s,Z_s,U_s)ds+\int_{t}^{T} d\mathcal{K}_s-\int_{0}^{t}\mathcal{Z}_s dB_s-\int_{0}^{t} \mathcal{U}_s dM_s,\quad t \in [0,T].
	$$
	Next, since $\L$ has finite left limits and using Corollary 3.11 in \cite{klimsiak2019reflected}, we deduce that the increasing process $\mathcal{K}$, satisfies
	$$
	\int_{0}^{T}\left(S_{u-}-\L_{u-}\right)d\mathcal{K}^{\ast}_u+\sum_{0 \leq u < T} (S_u-\L_u)\Delta_{+} \mathcal{K}_u=0.
	$$
	Moreover, from the definition of the process $(S_t)_{t \leq T}$ and the fact that $S_T=\xi \geq \L_T$ a.s., it is easy ro check that $S_t \geq \L_t$, $\forall t \in [0,T]$. Henceforth, the uniqueness of the solutions for RBSDE \eqref{basic equation lower} associated with $(\xi,f,\L)$ provided by Theorem \ref{Lower existence Thm} allows us to complete the proof.
\end{proof}

\begin{remark}
	Following a similar argument as the one used in the proof of Proposition \ref{Characetrization of lower obstacle} and assuming that conditions of Theorem \ref{existence and uniqueness for the one reflected} hold, we also provide a characterization for the solution of the BSDE \eqref{basic equation}  with one upper irregular barrier $\zeta$, represented as follows: Let $(Y_t,Z_t,K_t,U_t)_{t \leq T}$ be the solution of RBSDE \eqref{Reflected BSDE with one barrier}, then the state process $(Y_t)_{t \leq T}$ can be expressed as
	$$
	Y_t=\essinf_{\eta \in \mathcal{T}_{[t,T]}} \mathbb{E}\left[ \int_{t}^{\eta}f(s,Y_s,Z_s,U_s)ds+\zeta_{\eta} \mathds{1}_{\{\eta<T\}}+\xi \mathds{1}_{\{\eta=T\}} \mid \mathcal{F}_t \right],\quad t \in [0,T].
	$$
\end{remark}
\section{Comparison principal}
\label{sec 4}
First, we need to state the following auxiliary result, where the proof is presented similarly as in \cite[Proposition 3]{dumitrescu2018bsdes}.
\begin{proposition}
	Let $(\varphi_t)_{t \leq T}$ and $(\psi_t)_{t \leq T}$ be two $\mathbb{F}$-predictable real-valued processes, and let $(\Lambda_t)_{t \leq T}$ be the solution of the forward SDE:
	$$ 
	d\Lambda_t = \Lambda_{t-}(\varphi_t dB_t + \psi_t dM_t). 
	$$
	Assume that $\psi_t \geq -1$ and $\gamma_t dt \otimes d \mathbb{P}$-a.s., and further that the random variable $\int_{0}^{T} \{\varphi^2_s + \psi_s^2 \gamma_s\}ds$ is bounded. Then the process $(\Lambda_t)_{t \leq T}$ is a non-negative martingale satisfying $\mathbb{E}\left[\sup_{0 \leq t \leq T}|\Lambda_t|^2\right] < +\infty$.
	
	In particular, if there a third $\mathbb{F}$-predictable process $(\delta_t)_{t \leq T}$ such that $\int_{0}^{T} \left| \delta_s\right|ds$ is bounded. Then, the process $(\Lambda^{\star}_t)_{t \geq 0}$ solution of the following forward SDE:
	$$
	d\Lambda^{\star}_t=\Lambda^{\star}_{t-}\left(\delta_t dt+\varphi_t dB_t+\psi_t dM_t\right),
	$$
	is a non-negative special RCLL semimartingale satisfying  $\mathbb{E}\left[\sup_{0 \leq t \leq T}\left| \Lambda^{\star}_t\right|^2 \right]<+\infty$.
	\label{Integr problm solved}
\end{proposition}

\begin{remark}
	The process $(\Lambda^{\star}_t)_{t \leq T}$ defined in Proposition \ref{Integr problm solved},	satisfies the so called Dol\'{e}ans-Dade formula (see \cite[Section II.8]{Protter} for more details)
	\begin{equation*}
		\begin{split}
			\Lambda^{\star}_t=e^{\int_{0}^{t}\delta_s ds} \exp\left\{\int_{0}^{t}\varphi_s dB_s-\dfrac{1}{2}\int_{0}^{t}\left|\varphi_s \right|^2 ds \right\}e^{-\int_{0}^{t}\psi_s\gamma_sds}\left(1+\psi_{\tau}\mathds{1}_{\{\tau \leq t\}}\right),~t \in [0,T].
		\end{split}
	\end{equation*}
	In particular, when $\psi \geq -1$ (resp. $\psi>-1$), we derive that $\Lambda^{\star} \geq 0$ (resp. $\Lambda^{\star} > 0$).
	\label{remark of positivity}
\end{remark}

Next, it is important to note that in our framework the comparison theorem cannot  automatically be obtained, unlike the simpler case of Brownian filtration. Moving forward, we assume the following about the driver $f$ and the corresponding stochastic Lipschitz parameters: 
\begin{itemize}
	\item[\textbf{(H4)}]   $\int_{0}^{T}\left\{\mu_s+\theta^2_s+\left| \nu_s\right|^2  \gamma_s ds\right\}ds$ is a bounded random variable, and that there exists a map 
	$$
	\lambda : \Omega \times [0,T] \times \mathbb{R}^4 \rightarrow \mathbb{R};~ (\omega,t,y,z,u_1,u_2) \mapsto \lambda^{y,z,u_1,u_2}_t(\omega)
	$$
	$\mathcal{P}\otimes\mathcal{B}(\mathbb{R}^4)$-measurable with 
	$$
	\left|\lambda^{y,z,u_1,u_2}_t \right|^2 \leq \nu_t ~\text{ and }~~ \lambda^{y,z,u_1,u_2}_t \geq  -1,~d\mathbb{P}\otimes dt\text{-a.e.,} 
	$$
	satisfying $d\mathbb{P}\otimes dt$-a.s., for all $(y,z,u_1,u_2) \in \mathbb{R}^4$;
	\begin{equation}
		f(t,y,z,u_1)-f(t,y,z,u_2) \leq \lambda^{y,z,u_1,u_2}_t(u_1-u_2)\gamma_t.
		\label{Monotonicity}
	\end{equation}
\end{itemize}
\begin{remark}
	Note that, if \eqref{Monotonicity} is true we also have:
	$$
	f(t,y,z,u_1)-f(t,y,z,u_2) \geq \lambda^{y,z,u_2,u_1}_t(u_1-u_2)\gamma_t.
	$$
	It suffices to change the role of $u_1$ and $u_2$ in $\lambda$.
\end{remark}

Let $(Y^j,Z^j,K^j,U^j)$ be the unique solution of the RBSDE \eqref{basic equation lower} associated
with data $(\xi^j,f^j,\L^j)$, for $j = 1, 2$. Then we have the following comparison result:
\begin{theorem}
	Assume that:
	\begin{itemize}
		\item $\xi^1 \leq \xi^2$.
		\item $
		f^1\left(t, Y_t^2, Z_t^2,U^2_t\right) \leq f^2\left(t, Y_t^2, Z_t^2,U^2_t\right),~t \in [0,T],~d\mathbb{P}\otimes dt\text{-a.e.}
		$
		\item $\L^1_t \leq \L^2_t$, $\forall t \in [0,T]$ a.s.
	\end{itemize}
	Then, $Y^1_t \leq Y^2_t$, $\forall t \in [0,T]$ a.s.
	\label{Comparison theorem}
\end{theorem}

\begin{proof}
	Let us set $\bar{\mathfrak{G}}=\mathfrak{G}^1-\mathfrak{G}^2$ for $\mathfrak{G}^j \in \{Y^j,Z^j,K^j,U^j,\xi^j,\L^j\}$ and $j \in \{1,2\}$. Then a standard calculation allows us to obtain
	\begin{equation}
		\bar{Y}_t=\bar{\xi}+\int_{t}^{T}\left(\delta_s\bar{Y}_{s}+\varphi_s \bar{Z}_s+\psi_s+\phi_s\right)ds+\left(\bar{K}_T-\bar{K}_t\right)-\int_{t}^{T}\bar{Z}_s dB_s-\int_{t}^{T}\bar{U}_s dM_s,
		\label{IMN DYN}
	\end{equation}
	with 
	\begin{itemize}
		\item $\delta_s=\left(\bar{Y}_{s-}\right)^{-1}\mathds{1}_{\{\bar{Y}^{}_{s-} \neq 0\}}\left(f^1(s,Y^1_{s-},Z^1_s,U^1_s)-f^1(s,Y^2_{s-},Z^1_s,U^1_s)\right)$,
		\item $\varphi_s=\left(\bar{Z}_s\right)^{-1}\mathds{1}_{\{\bar{Z}_{s-} \neq 0\}}\left(f^1(s,Y^2_{s-},Z^1_s,U^1_s)-f^1(s,Y^2_{s-},Z^2_s,U^1_s)\right)$,
		\item  $\psi_s=f^1(t,Y^2_{t-},Z^2_t,U^1_t)-f^1(t,Y^2_{t-},Z^2_t,U^2_t)$,
		\item $\phi_s=f^1(t,Y^2_{t-},Z^2_t,U^2_t)-f^2(t,Y^2_{t-},Z^2_t,U^2_t)$.
	\end{itemize}
	
	By definition, the process $(\varphi_t)_{t \leq T}$ is $\mathcal{P}$-measurable and satisfies  $\left| \varphi_t\right| \leq  \theta_t$, $\forall t \in [0,T]$ a.s. Therefore, $\int_{0}^{T}\left| \varphi_t\right|^2dt$ is a bounded $\mathcal{F}_T$-measurable random variable. Next, using the assumption \textbf{(H4)}, we derive the existence of a $\mathcal{P}\otimes \mathcal{B}(\mathbb{R}^d)$-measurable mapping $\lambda:\Omega \times [0,T]\times \mathbb{R}^4;(\omega,t,y,z,u_1,u_2)\mapsto\lambda_t^{y,z,u^1,u^2}(\omega)$ such that $\lambda_t^{y,z,u^1,u^2} \geq -1$ and $
	\psi_t \leq \lambda_t^{Y^2_t,Z^2_t,U^1_t,U^2_t}\bar{U}_t\gamma_t
	$.
	
	Now, we defined the process $(\Lambda_{t})_{t \in [0,T]}$ as the unique solution of the following forward SDE:
	$$
	d\Lambda_{t}=\Lambda_{t-}\left(\varphi_t dB_t+\lambda_t^{Y^2_t,Z^2_t,U^1_t,U^2_t}dM_t\right),\quad \Lambda_{0}=1.
	$$
	Note that as, $\int_{0}^{T}\big\{\varphi^2_s+\big| \lambda_s^{Y^2_s,Z^2_s,U^1_s,U^2_s}\big|^2\gamma_s \big\}ds$ is a bounded random variable, then from Proposition \ref{Integr problm solved}, we deduce that $\mathbb{E}\sup_{0 \leq t \leq T}\left| \Lambda_{t}\right|^2<+\infty $ and $\Lambda \geq  0$. Henceforth, $\Lambda_{t,\cdot}$ defines a positive square integrable martingale. Using Proposition 3.1 in \cite{kusuoka1999remark}, we deduce that there exists a probability measure $\mathbb{Q}$ such that $\bar{B}_t:=B_t-\int_{0}^{t} \varphi_s ds$ is a one-dimensional Brownian motion and $\bar{M}_t:=M_t-\int_{0}^{t}\lambda_s^{Y^2_s,Z^2_s,U^s_t,U^2_s} \gamma_sds $ is an $(\mathbb{F},\mathbb{Q})$-martingale with default jump, where the new equivalent probability measure $\mathbb{Q}$ is defined from the process $\left(\Lambda_t\right)_{t \leq T}$ by the so called Dol\'{e}ans-Dade formula, as follows:
	\begin{equation*}
		\begin{split}
			d\mathbb{Q}:=&\exp\left\{\int_{0}^{t}\varphi_s dB_s-\dfrac{1}{2}\int_{0}^{t}\left|\varphi_s \right|^2ds \right\}\exp\left\{-\int_{0}^{t}\lambda_s^{Y^2_s,Z^2_s,U^1_s,U^2_s}\gamma_sds\right\}\\
			&\quad\left(1+\lambda_{\tau}^{Y^2_\tau,Z^2_\tau,U^1_\tau,U^2_\tau}\mathds{1}_{\{\tau \leq t\}}\right)d\mathbb{P}.
		\end{split}
	\end{equation*}
	Then, from \eqref{IMN DYN}, we deduce that
	\begin{equation}
		\bar{Y}_t=\bar{\xi}+\int_{t}^{T}\left(\delta_s \bar{Y}_{s-}+\bar{\psi}_s+\phi_s\right)ds+\left(\bar{K}^{}_T-\bar{K}^{}_t\right)-\int_{t}^{T}\bar{Z}_s d\bar{B}_s-\int_{t}^{T}\bar{U}_s d\bar{M}_s,
		\label{Strong optional semimartingale Tanaka}
	\end{equation}
	with $
	\bar{\psi_s}:=\psi_s-\lambda_s^{Y^2_s,Z^2_s,U^s_t,U^2_s}\bar{U}_s\gamma_s$ for $s \in [0,T]$
	
	By applying Lemma \ref{Tanaka-type Formula} to the convex function $\Phi(x) = x^{+}$ and the $\mathbb{F}$-optional semimartingale (\ref{Strong optional semimartingale Tanaka}) with  $\bar{Y}^{\ast}_t=\bar{\xi}+\int_{t}^{T}\left(\delta_s\bar{Y}_{s-}+\bar{\varphi}_s +\phi_s\right)ds+\left(\bar{K}^{\ast}_T-\bar{K}^{\ast}_t\right)-\int_{t}^{T}\bar{Z}_s dB_s-\int_{t}^{T}\bar{U}_s dM_s$, $\Delta_{+} \bar{Y}_s=-\Delta_{+} \bar{K}_s$, we find that there exists a non-decreasing process $\left(\bar{\mathcal{L}}_t\right)_{t \leq T}$ with regulated trajectories such that:
	\begin{equation}
		\begin{split}
			\bar{Y}^{+}_t=\bar{Y}^{+}_0&-\int_{0}^{t}\mathds{1}_{\{Y^1_{s}>Y^2_{s}\}}\left(\delta_s \bar{Y}_{s}+\bar{\psi}_s+\phi_s\right)ds+\int_{0}^{t}\mathds{1}_{\{Y^1_{s-}>Y^2_{s-}\}} d \bar{K}^{\ast}_s\\
			&\qquad+\int_{0}^{t}\mathds{1}_{\{Y^1_{s}>Y^2_{s}\}}\bar{Z}_s d\bar{B}_s+\int_{0}^{t}\mathds{1}_{\{Y^1_{s-}>Y^2_{s-}\}}\bar{U}_s d\bar{M}_s+\bar{\mathcal{L}}_t,\quad t \in [0,T].
		\end{split}
		\label{Y plus dynamic}
	\end{equation}
	Now, applying Theorem \ref{Ito's formula Theorem} to the dynamics of the process $\left(\bar{Y}^{+}_t\right)_{t \leq T}$ defined by (\ref{Y plus dynamic}), we obtain, for any stopping time $\eta \in \mathcal{T}_{[0,T]}$,
	\begin{equation}
		\begin{split}
			&e^{\beta A_{\eta}} \left| \bar{Y}^{+}_\eta \right|^2+\beta \int_{\eta}^{T} e^{\beta A_s} \left| \bar{Y}^{+}_s \right|^2 dA_s\\
			&   \leq e^{\beta A_T} \left|\bar{ \xi}^{+}_T  \right|^2+2 \int_{\eta}^{T} e^{\beta A_s} \bar{Y}^{+}_{s}\left(\delta_s \bar{Y}_{s}+\bar{\psi}_s+\phi_s\right) ds\\
			&\quad +2\int_{\eta}^{T} e^{\beta A_s }\bar{Y}^{+}_{s-}\mathds{1}_{\{Y^1_{s-}>Y^2_{s-}\}}d\bar{K}^{\ast}_s +2\sum_{\eta \leq s < T} e^{\beta A_s} \bar{Y}^{+}_s \mathds{1}_{\{Y^1_{s}>Y^2_{s}\}} \Delta_{+}\bar{K}_s\\
			&\quad-2 \int_{\eta}^{T} e^{\beta A_s+\lambda s} \bar{Y}^{+}_{s}\bar{Z}_sd\bar{B}_s-2\int_{\eta}^{T} e^{\beta A_s} \bar{Y}^{+}_{s-}\bar{U}_s d\bar{M}_s\\
			&\quad-\sum_{\eta <s \leq T}e^{\beta A_s} \left| \Delta_{-} \bar{Y}^{+}_s\right|^2 -\sum_{\eta \leq s < T} e^{\beta A_s+\lambda s} \left|\Delta_{+}\bar{Y}^{+}_s \right|^2-\int_{\eta}^{T}e^{\beta A_s }\bar{Y}^{+}_{s-}d\bar{\mathcal{L}}_s,
		\end{split}
		\label{both sides}
	\end{equation}
	Given the minimality condition in the RBSDE (\ref{Reflected BSDE with one barrier}) for the reflection processes $K^1$ and $K^2$, along with  assumption $\L^1 \leq \L^2$ and Lemma \ref{Reflecting property}, we can deduce
	\begin{equation}
		\mathds{1}_{\{Y^1_{s-}>Y^2_{s-}\}}d\bar{K}^{\ast}_s =\mathds{1}_{\{Y^1_{s-}>Y^2_{s-}\}}\left(dK^{1,\ast}_s-dK^{2,\ast}_s\right)= -\mathds{1}_{\{Y^1_{s-}>Y^2_{s-}\}}dK^{2,\ast}_s\leq 0,
		\label{minimality condition K ast}
	\end{equation}
	and
	\begin{equation}
		\mathds{1}_{\{Y^1_{s}>Y^2_{s}\}} \Delta_{+}\bar{K}_s=\mathds{1}_{\{Y^1_{s}>Y^2_{s}\}} \left(\Delta_{+}K^1_s- \Delta_{+}K^2_s\right)= -\mathds{1}_{\{Y^1_{s}>Y^2_{s}\}} \Delta_{+}K^2_s\leq 0.
		\label{minimality condition K right jumps}
	\end{equation}
	Returning to \eqref{both sides}, we can exploit the negativity of $\bar{\xi}$, $\phi$ and $\bar{\psi}$, the fact that $ \delta_t  \leq \mu_t \leq \alpha_t^2$ and utilize inequalities \eqref{minimality condition K ast} and \eqref{minimality condition K right jumps}, by then taking the expectation on both sides with respect to the new measure $\mathbb{Q}$ denoted by $\bar{\mathbb{E}}$, we obtain for all $t \in [0,T]$ and any $\beta>2$
	$$
	\bar{\mathbb{E}}e^{\beta A_{\eta}} \left| \bar{Y}^{+}_\eta \right|^2 \leq 0.
	$$
	It follows that $\bar{Y}^{+}_\eta=0$ for any $\eta \in \mathcal{T}_{[0,T]}$. Then, as the process $\bar{Y}$ is optional and using the cross section theorem, we deduce that $\bar{Y}_t \leq 0$, $\forall t \in [0,T]$ $\mathbb{Q}$-a.s. and so $Y^1_t \leq Y^2_t$, $\forall t \in [0,T]$ $\mathbb{P}$-a.s.
\end{proof}

\begin{remark}
	\begin{itemize}
		\item If $\L^j\equiv-\infty$ or $\zeta^j\equiv+\infty$ for $j \in \{1,2\}$, then $dK^j=0$ for $j \in \{1,2\}$ and the comparison theorem is still applied to the classical BSDE \eqref{basic equation : classical BSDE} as well.
		\item In the case where the generator $f$ does not depend on the $u$ variable, the comparison result \ref{Comparison theorem} still holds without the need to assume the \textbf{(H4)} assumption.
	\end{itemize}
	\label{Comparison BSDE}
\end{remark}

\section{BSDE with irregular barrier and a related optimal stopping problem with $\mathcal{E}^f$-expectations}
\label{sec 5}
We first recall some definitions needed in the current section. We point out that here and through the current part we assume that given a terminal time $0<T <+\infty$ and a generator $f$ that satisfies conditions \textbf{(H2)} and \textbf{(H4)} with a $(-1,\infty)$-valued $\mathcal{P}\otimes\mathcal{B}(\mathbb{R}^4)$-measurable mapping $\lambda$.\\
We start with the so called $f$-conditional expectation, $\mathcal{E}^f$-expectation or $f$-evaluation in the terminology of Peng \cite{Peng2004}. In this section, we choose to work with the term $\mathcal{E}^f$-expectation.
\begin{definition}[$\mathcal{E}^f$-expectation]
	Let $T^{\prime} \in [0,T]$ be a deterministic time. Let $\xi \in \mathbb{L}^2_{\beta}(\mathcal{F}_{T'})$. The $\mathcal{E}^f$-expectation of $\xi$ denoted by $\left(\mathcal{E}^f_{t,T'}(\xi)\right)_{t \in [0,T']}$ is defined as the the first component of the BSDE \eqref{basic equation : classical BSDE} with default jump associated with terminal time $T'$ and data $(\xi,f)$. More generally, for each time $T'\in [0,T]$ and each terminal condition $\xi \in \mathbb{L}^2_{\beta}(\mathcal{F}_{T'})$, due to Theorem \ref{Existence and uniqueness unconstrained BSDE}, we may define an operator $\mathcal{E}^f: \left(T',\xi\right) \to \mathcal{E}^f_{\cdot,T'}(\xi)$, and we say that $\mathcal{E}^f_{\cdot,T'}(\xi)$ is the $\mathcal{E}^f$-expectation process of $\xi$.
\end{definition}

\begin{remark}
	\begin{itemize}
		\item The previous operator can be generalized to the case of stopping times $\eta$ in the class $\mathcal{T}_{[0,T]}$. Recall that $Y$ is a solution of the BSDE with terminal time $\eta \in \mathcal{T}_{[0,T]}$, terminal value $\xi$ and coefficient $f$ if $Y=Y^{\star}$ where $Y^{\star}$ is the solution of the BSDE \eqref{basic equation : classical BSDE} associated with data $(\xi,f\mathds{1}_{\rrbracket0,\eta\rrbracket})$ (see also \cite[Lemma 3.3]{topolewski2018reflected} for a related study).
		\item A process $Y \in \mathcal{S}^2_{\beta}$ is a strong $\mathcal{E}^f$-martingale on $[\sigma,\eta]$ with $\sigma,\eta \in \mathcal{T}_{[0,T]}$ and $\sigma \leq \eta$ a.s., if and only if, $Y=Y^{\star}$ on $[\sigma,\eta]$, where $Y^{\star}$ is the solution of the BSDE \eqref{basic equation : classical BSDE} associated with terminal time $\eta$ and data $(Y_{\eta},f)$.
	\end{itemize}
	\label{Remark of the ending stp}
\end{remark}
Following this remark, we next gives the notion of strong $\mathcal{E}^f$-(sub,  super)martingale in the context of driver $f$ satisfying a stochastic Lipschitz property.
\begin{definition}
	Let $Y \in \mathcal{S}^2_{\beta}$.
	\begin{itemize}
		\item  The process $Y$ is said to be a strong $\mathcal{E}^f$-supermartingale (resp. $\mathcal{E}^f$-submartingale), if $\mathcal{E}^f_{\sigma,\eta}(Y_{\eta}) \leq Y_\sigma$ (resp. $\mathcal{E}^f_{\sigma,\eta}(Y_{\eta}) \geq Y_\sigma$) a.s. on $\sigma \leq \eta$, for all $\sigma,\eta \in \mathcal{T}_{[0,T]}$.
		\item The process $Y$ is said to be a strong $\mathcal{E}^f$-martingale, if it is both a strong $\mathcal{E}^f$-supermartingale and a strong $\mathcal{E}^f$-submartingale.
	\end{itemize}
	Let $\sigma,\eta \in \mathcal{T}_{[0,T]}$.
	\begin{itemize}
		\item The process $Y$ is said to be a strong $\mathcal{E}^f$-supermartingale (resp. $\mathcal{E}^f$-submartingale) on $[\sigma,\eta]$, if for all $\sigma^{\star}, \eta^{\star} \in \mathcal{T}_{[0,T]}$ such that $\sigma \leq \sigma^{\star} \leq \eta^{\star} \leq \eta$ a.s., we have $\mathcal{E}^f_{\sigma^{\ast},\eta^{\ast}}(Y_{\eta^{\ast}}) \leq Y_{\sigma^{\ast}}$ (resp. $\mathcal{E}^f_{\sigma^{\ast},\eta^{\ast}}(Y_{\eta^{\ast}}) \geq Y_{\sigma^{\ast}}$) a.s. Finally, the notion of $\mathcal{E}^f$-martingale on $[\sigma,\eta]$ is defined similarly.
	\end{itemize}
	\label{definition strong on}
\end{definition}

Finally, we present the definition of right-upper semi-continuity for progressively measurable processes.
\begin{definition}
	A progressiveness measurable process $\left(\L_t\right)_{t \leq T}$ is said to be right-upper semi-continuous (r.u.s.c for short) along stopping times, if for all $\eta \in \mathcal{T}_{[0,T]}$ and each non-increasing sequence of stopping times $\{\eta_n\}_{n \in \mathbb{N}} \subset \mathcal{T}_{[0,T]}$ such that $\eta_n \downarrow \eta$ a.s., we have $\L_{\eta} \geq \limsup_{n \to+\infty} \L_{\eta_n}$ a.s.
\end{definition}
\begin{remark}
	If $\left(\L_t\right)_{t \leq T}$ is a progressiveness measurable regulated process, the right-upper semi-continuity is equivalent to  $\L_{\eta} \geq  \L_{\eta+}$ for all $\eta \in \mathcal{T}_{[0,T]}$. 
	\label{Remark on left semi-continuity}
\end{remark}

\begin{lemma}
	Let $Y \in \mathcal{S}^2_{\beta}$ be a strong $\mathcal{E}^f$-supermartingale on $[\sigma,\eta]$ with $\sigma,\eta \in \mathcal{T}_{[0,T]}$ such that $\sigma \leq \eta$ a.s. The following two properties are equivalent:
	\begin{itemize}
		\item[(i)] $Y$ is a strong $\mathcal{E}^f$-martingale on $[\sigma,\eta]$.
		\item[(ii)] $Y_{\sigma}=\mathcal{E}^f_{\sigma,\eta}(Y_{\eta})$.
	\end{itemize}
\end{lemma}
\begin{proof}
	\begin{itemize}
		\item Property (i) implies Property (ii):\\
		Follows from Definition \ref{definition strong on}.
		\item Property (ii) implies Property (i):\\
		Let $\sigma^{\star} \in \mathcal{T}_{[\sigma,\eta]}$. From (ii), we have $Y_{\sigma}=\mathcal{E}^f_{\sigma,\eta}(Y_{\eta})$. As $\mathcal{E}^f_{\cdot,\eta}(Y_{\eta})$ is a solution of the \eqref{basic equation : classical BSDE} on $[0,\eta]$, we deduce that $\mathcal{E}^f_{\sigma,\eta}(Y_{\eta})=\mathcal{E}^f_{\sigma,\sigma^{\star}}\left(\mathcal{E}^f_{\sigma^{\star},\eta}(Y_{\eta})\right)=Y_{\sigma}$. Now, using the comparison principal (see Theorem \ref{Comparison for BSDEs}-(i)), we deduce from the property $Y_{\sigma^{\ast}} \geq \mathcal{E}^f_{\sigma^{\ast},\eta}(Y_{\eta})$ (by assumption) that $\mathcal{E}^f_{\sigma,\sigma^{\star}}\left(Y_{\sigma^{\star}}\right)\geq \mathcal{E}^f_{\sigma,\sigma^{\star}}\left(\mathcal{E}^f_{\sigma^{\star},\eta}(Y_{\eta})\right) $. Thus, $Y_{\sigma} \geq \mathcal{E}^f_{\sigma,\sigma^{\star}}\left(Y_{\sigma^{\star}}\right) \geq  \mathcal{E}^f_{\sigma,\sigma^{\star}}\left(\mathcal{E}^f_{\sigma^{\star},\eta}(Y_{\eta})\right) =Y_{\sigma} $ and then $\mathcal{E}^f_{\sigma,\sigma^{\star}}\left(Y_{\sigma^{\star}}\right) =\mathcal{E}^f_{\sigma,\sigma^{\star}}\left(\mathcal{E}^f_{\sigma^{\star},\eta}(Y_{\eta})\right)=Y_{\sigma}$. Now, as $Y_{\sigma^{\star}} \geq \mathcal{E}^f_{\sigma^{\star},\eta}(Y_{\eta})$ and using the strict version of the comparison principal (see Theorem \ref{Comparison for BSDEs}-(ii)), we deduce that $Y_{\sigma^{\star}} = \mathcal{E}^f_{\sigma^{\star},\eta}(Y_{\eta})$.
	\end{itemize}
The the proof is complete.
\end{proof}
\begin{remark}
	Let $(Y,Z,K,U)$ be the unique a solution of the RBSDE \eqref{basic equation lower} (under conditions of Theorem \ref{Lower existence Thm}) in the sense of Definition \ref{Definition of the solution}. From the Theorem \ref{Comparison theorem for GBSDE}, we deduce that the state process $Y$ is a strong $\mathcal{E}^f$-supermartingale. 
	\label{Strong supermartingale property}
\end{remark}

Here and afterward, in order to make the study more simple, we assume that the terminal value of a given RBSDE and the terminal value of the corresponding obstacle coincide at the terminal time. More precisely, for a given terminal condition $\xi$ and an obstacle $\L$, we set that $\xi=\L_T$ and $\L_T \in \mathbb{L}^2_{\beta}$.
\begin{lemma}
	Let $\L$ be a process satisfying condition \textbf{(H3')} with $\L \in \mathcal{S}^2_{2 \beta}$. Assume moreover that $\L$ is r.u.s.c along stopping times.\\
	Let $(Y,Z,K,U)$ be the unique solution of the RBSDE \eqref{basic equation lower} associated with $(\L_T,f,\L)$ in the sense of Definition \ref{Definition of the solution}, which exists in virtue of Theorem \ref{Lower existence Thm}. Let $\varepsilon>0$ and $\sigma \in \mathcal{T}_{[0,T]}$. Let $\eta_{\sigma}^{\varepsilon}$ be defined by
	\begin{equation}
		\eta_{\sigma}^{\varepsilon}:=\inf\left\{t \geq \sigma : Y_t \leq \L_t+\varepsilon\right\}.
		\label{Optila varep Stp}
	\end{equation}
	Then, the following properties holds:
	\begin{itemize}
		\item[(i)] $Y_{\eta_{\sigma}^{\varepsilon}} \leq \L_{\eta_{\sigma}^{\varepsilon}}+\varepsilon$ a.s.
		\item[(ii)]  The state process is a strong $\mathcal{E}^f$-martingale on $[\sigma,\eta_{\sigma}^{\varepsilon}]$.
	\end{itemize}
	\label{Lemma needed for optimal stopping}
\end{lemma}

\begin{proof}
	With a few clear changes, the proof is comparable to that in \cite[Lemma 4.1]{Grigorova2017} for Poisson random measure filtering and Brownian motion. We provide it in our context for the convenience of the reader.
	\begin{itemize}
		\item[(i)] First, note that, from \cite[Theorem 3.11]{He}, we derive that $\eta_{\sigma}^{\varepsilon}$ is the d\'{e}but after $\sigma$ of a progressive set, we deduce that $\eta_{\sigma}^{\varepsilon}$ is an $\mathbb{F}$-stopping time in $[\sigma,T]$ as $Y_T=\L_T \leq \L_T+\varepsilon$, i.e.  $\eta_{\sigma}^{\varepsilon} \in \mathcal{T}_{[\sigma,T]}$.\\
		Next, we suppose that $\mathbb{P}\left(Y_{\eta_{\sigma}^{\varepsilon}} > \L_{\eta_{\sigma}^{\varepsilon}}+\varepsilon \right)>0$. From Remark \ref{Reflecting property}, we deduce that $\Delta_{+}K_{\eta_{\sigma}^{\varepsilon}}=0$ on the set $\left\{Y_{\eta_{\sigma}^{\varepsilon}} > \L_{\eta_{\sigma}^{\varepsilon}}+\varepsilon\right\}\subset\left\{Y_{\eta_{\sigma}^{\varepsilon}} > \L_{\eta_{\sigma}^{\varepsilon}}\right\}$. On the other hand, from BSDE \eqref{basic equation lower}-(i), we know that $\Delta_{+}Y_{\eta_{\sigma}^{\varepsilon}}=-\Delta_{+} K_{\eta_{\sigma}^{\varepsilon}}=K_{\eta_{\sigma}^{\varepsilon}}-K_{\eta_{\sigma}^{\varepsilon}+}$. Thus, $Y_{\eta_{\sigma}^{\varepsilon}+}=Y_{\eta_{\sigma}^{\varepsilon}}$ on the set $\left\{Y_{\eta_{\sigma}^{\varepsilon}} > \L_{\eta_{\sigma}^{\varepsilon}}+\varepsilon\right\}$. Following this we derive that $Y_{\eta_{\sigma}^{\varepsilon}+} > \L_{\eta_{\sigma}^{\varepsilon}}+\varepsilon$ on the set $\left\{Y_{\eta_{\sigma}^{\varepsilon}} > \L_{\eta_{\sigma}^{\varepsilon}}+\varepsilon\right\}$.
		Now, let's fix $\omega \in \Omega$, from the definition of the time $\eta_{\sigma}^{\varepsilon}(\omega)$, we  get the existence of a sequence $\{\mathfrak{t}_n(\omega)\}_{n \in \mathbb{N}}$ such that $Y_{\mathfrak{t}_n(\omega)} \leq \L_{\mathfrak{t}_n(\omega)}+\varepsilon$ $\forall n \in \mathbb{N}$ and $\mathfrak{t}_n(\omega) \downarrow \eta_{\sigma}^{\varepsilon}(\omega)$ as $n \to +\infty$. As the processes have finite right limits, we deduce by taking the limit on both sides when $n \to +\infty$ that $Y_{\eta_{\sigma}^{\varepsilon}(\omega)+} \leq \L_{\eta_{\sigma}^{\varepsilon}(\omega)+}+\varepsilon$ and from the assumption that $\L$ is r.u.s.c along stopping times and Remark \ref{Remark on left semi-continuity}, we get $Y_{\eta_{\sigma}^{\varepsilon}(\omega)+} \leq \L_{\eta_{\sigma}^{\varepsilon}(\omega)}+\varepsilon$ a.s., which is a contraction. Hence, we deduce that $Y_{\eta_{\sigma}^{\varepsilon}} \leq \L_{\eta_{\sigma}^{\varepsilon}}+\varepsilon$ a.s. and the statement (i) of Lemma \ref{Lemma needed for optimal stopping} follows.
		\item[(ii)] First, from Definition \ref{regulated processes}, recall that the reflection process $K$ can be decomposed into three parts $K=K^g+K^c+K^d=K^{\ast}+K^g$, with the property that $\int_{0}^{T}\left(Y_s-\L_s\right)dK^c_s=0$. Meaning that the continuous part increases only when $Y$ touch the barrier at a continuous point an try to prevent it (refer to \cite[Remark 4.1]{ElOtmani2009} for more details concerning the reflection of the right-continuous part). On the other hand, for almost every $\omega \in \Omega$, we have $Y_t >\L_t+\varepsilon$ for all $t \in [\sigma,\eta_{\sigma}^{\varepsilon}[$. Then, $dK^c_t(\omega)=0$ in  $ [\sigma,\eta_{\sigma}^{\varepsilon}[$, by continuity, we deduce that $dK^c_t(\omega)=0$ in  $ [\sigma,\eta_{\sigma}^{\varepsilon}]$, in other word, the function $t \mapsto K^c_t(\omega)$ is constant on the time interval $ [\sigma,\eta_{\sigma}^{\varepsilon}]$. Furthermore, note that $Y_{t-} \geq \L_{t-}+\varepsilon>\L_{t-}$ $\forall t \in [\sigma,\eta_{\sigma}^{\varepsilon}[$, then from Remark \ref{Reflecting property}, we deduce that $\Delta_{-}K^{\ast}_{t}=0$, $\forall t \in [\sigma,\eta_{\sigma}^{\varepsilon}[$. Now, from the definition of $\eta_{\sigma}^{\varepsilon}$ and the fact that $Y_{\eta_{\sigma}^{\varepsilon}-} \geq \L_{\eta_{\sigma}^{\varepsilon}-}+\varepsilon>\L_{\eta_{\sigma}^{\varepsilon}-}$, we also derive that $\Delta_{-}K^{\ast}_{t}=0$, $\forall t \in [\sigma,\eta_{\sigma}^{\varepsilon}]$. Finally, a similar study using Remark \ref{Reflecting property}, allows to conclude that $\Delta_{+}K=0$ on $[\sigma,\eta_{\sigma}^{\varepsilon}[$. Hence, the process $K^g$ is constant on $[\sigma,\eta_{\sigma}^{\varepsilon}[$. By the left-continuity of the purely-discontinuous process $K^g=\sum_{0 \leq s <t}\Delta_{+}K_s$, we derive that $t \mapsto K^g_t(\omega)$ is constant on $[\sigma,\eta_{\sigma}^{\varepsilon}]$. Resuming all this, we obtain that $t \mapsto K_t$ is constant on $[\sigma,\eta_{\sigma}^{\varepsilon}]$. Hence, $(Y,Z,U)$ is a solution of the BSDE \eqref{basic equation : classical BSDE} on $[\sigma,\eta_{\sigma}^{\varepsilon}]$ associated with terminal time $\eta_{\sigma}^{\varepsilon}$, terminal condition $Y_{\eta_{\sigma}^{\varepsilon}}$ and driver $f$. The result is obtained using Remark \ref{Remark of the ending stp} and we can the write $Y_{\sigma^{\ast}}=\mathcal{E}^f_{\sigma^{\ast},\eta_{\sigma}^{\varepsilon}}(Y_{\eta_{\sigma}^{\varepsilon}})$ for all $\sigma^{\ast} \in \mathcal{T}_{[0,T]}$ such that $\sigma \leq \sigma^{\ast} \leq \eta_{\sigma}^{\varepsilon}$ a.s.
	\end{itemize}
\end{proof}

We also have to state the following auxiliary lemma:
\begin{lemma}
	Let $\sigma \in \mathcal{T}_{[0,T]}$, $\eta \in \mathcal{T}_{[\sigma,T]}$ and $\L \in \mathcal{S}^2_{\beta}$, then 
	$$
	\left| \mathcal{E}^f_{\sigma,\eta_{\sigma}}\left(\L_{\eta_{\sigma}}+\varepsilon\right)-\mathcal{E}^f_{\sigma,\eta_{\sigma}}\left(L_{\eta_{\sigma}}\right) \right| \leq\mathfrak{c}_{\beta,A,T} \varepsilon,
	$$
	where $\mathfrak{c}_{\beta,T}$ is a positive constant that relies only on the terminal time $T$, the boundedness constant of the random variable $A$ and $\beta$. In particular, we have
	$$
	\mathcal{E}^f_{\sigma,\eta_{\sigma}}\left(\L_{\eta_{\sigma}}+\varepsilon\right)\leq \mathcal{E}^f_{\sigma,\eta_{\sigma}}\left(L_{\eta_{\sigma}}\right)+\mathfrak{c}_{\beta,A,T} \varepsilon.
	$$
	\label{Also needed}
\end{lemma}
\begin{proof}
	From Proposition \ref{Preliminalry estimaes}, we have
	\begin{equation*}
		\begin{split}
			\left| \mathcal{E}^f_{\sigma,\eta_{\sigma}^{\varepsilon}}\left(\L_{\eta_{\sigma}^{\varepsilon}}+\varepsilon\right)-\mathcal{E}^f_{\sigma,\eta_{\sigma}^{\varepsilon}}\left(L_{\eta_{\sigma}^{\varepsilon}}\right) \right|^2 &\leq c_{\beta} \mathbb{E} \left[e^{\beta A_{T}}\left|\left(\L_{\eta_{\sigma}^{\varepsilon}}+\varepsilon\right)-\L_{\eta_{\sigma}^{\varepsilon}} \right|^2 \mid \mathcal{F}_{\sigma}\right]\\
			&= c_{\beta}\varepsilon^2 \mathbb{E}\left[ e^{\beta A_T} \mid \mathcal{F}_{\sigma} \right]  .
		\end{split}
	\end{equation*}
	Then the result follows from the boundedness of the random variable $A_T$ asserted in the assumption \textbf{(H4)}.
\end{proof}

Let's now outline the main result of this section.
\begin{theorem}
	Let $\left(\L_t\right)_{t \leq T}$ be a regulated process which is r.u.s.c along stooping satisfying condition \textbf{(H3')} with $\L \in \mathcal{S}^2_{2\beta}$. Let $\left(Y_t,Z_t,K_t,U_t\right)_{t \leq T}$ be the unique solution of the RBSDE \eqref{basic equation lower} associated with $\left(\L_T,f,\L\right)$ which exists due to Theorem \ref{Lower existence Thm} in the sense of Definition \ref{Definition of the solution}.  Then,
	\begin{itemize}
		\item[(i)] For each $\sigma \in \mathcal{T}_{[0,T]}$, the state process is a solution of the following optimal stopping problem with $\mathcal{E}^f$-expectation:
		\begin{equation}
			Y_{\sigma}=\esssup_{\eta\in \mathcal{T}_{[\sigma,T]}} \mathcal{E}^f_{\sigma,\eta}\left(\L_{\eta}\right)~\text{ a.s.}
			\label{Stopping problem}
		\end{equation}
		\item[(ii)] \textit{$(\beta,A,\varepsilon)$-optimality}: For every $\sigma \in \mathcal{T}_{[0,T]}$ and each $\varepsilon>0$, the stopping time \eqref{Optila varep Stp} is $(\beta,A,\varepsilon)$-optimal for the stopping problem with $\mathcal{E}^f$-expectation \eqref{Stopping problem}, in the sense that
		\begin{equation}
			Y_{\sigma} \leq \mathcal{E}^f_{\sigma,\eta_{\sigma}^{\varepsilon}}\left(\L_{\eta_{\sigma}^{\varepsilon}}\right)+\mathfrak{c}_{\beta,A,T}\varepsilon~\text{ a.s.}
			\label{Optionality relation}
		\end{equation}
		where the positive constant $\mathfrak{c}_{\beta,A,T}$ is given in the proof of Lemma \ref{Also needed}.
	\end{itemize}
	\label{Main result of sec 5}
\end{theorem}
\begin{proof}
	Let $\varepsilon>0$ and $\eta \in \mathcal{T}_{[\sigma,T]}$. From Remark \ref{Strong supermartingale property}, we deduce that $Y$ is a strong $\mathcal{E}^f$-supermartingale. Hence, we have $Y_{\sigma} \geq \mathcal{E}^f_{\sigma,\eta}(Y_{\eta})$. Moreover, from \eqref{basic equation lower}-(ii), we have  $Y_{\eta} \geq \L_{\eta}$ and from Theorem \ref{Comparison for BSDEs}, we have $\mathcal{E}^f_{\sigma,\eta}(Y_{\eta}) \geq \mathcal{E}^f_{\sigma,\eta}(\L_{\eta})$ , then $Y_{\sigma} \geq \mathcal{E}^f_{\sigma,\eta}(\L_{\eta})$. From the definition of the essential supremum, we deduce that $Y_{\sigma} \geq \esssup_{\eta\in \mathcal{T}_{[\sigma,T]}} \mathcal{E}^f_{\sigma,\eta}\left(\L_{\eta}\right)$ a.s. Showing the opposite disparity is still necessary. This will be done with the help of Remark \ref{Strong supermartingale property} together with the result of Lemma \ref{Lemma needed for optimal stopping}-(ii), which implies that $Y_{\sigma}=\mathcal{E}^f_{\sigma,\eta_{\sigma}^{\varepsilon}}\left(Y_{\eta_{\sigma}^{\varepsilon}}\right)$ a.s. Now, using the result of Lemma \ref{Lemma needed for optimal stopping}-(i) and Theorem \ref{Comparison for BSDEs}, we derive $Y_{\sigma}=\mathcal{E}^f_{\sigma,\eta_{\sigma}^{\varepsilon}}\left(Y_{\eta_{\sigma}^{\varepsilon}}\right) \leq \mathcal{E}^f_{\sigma,\eta_{\sigma}^{\varepsilon}}\left(\L_{\eta_{\sigma}^{\varepsilon}}+\varepsilon\right) $ a.s. Then, a direct application of Lemma \ref{Also needed} allows us to write $Y_{\sigma}=\mathcal{E}^f_{\sigma,\eta_{\sigma}^{\varepsilon}}\left(Y_{\eta_{\sigma}^{\varepsilon}}\right) \leq \mathcal{E}^f_{\sigma,\eta_{\sigma}^{\varepsilon}}\left(\L_{\eta_{\sigma}^{\varepsilon}}\right)+\mathfrak{c}_{\beta,A,T}\varepsilon$, which proofs the statement (ii) of Theorem \ref{Main result of sec 5}, i.e. the $(\beta,A,\varepsilon)$-optimality provided by relation \eqref{Optionality relation}. Now, from inequality  $Y_{\sigma} \leq \mathcal{E}^f_{\sigma,\eta_{\sigma}^{\varepsilon}}\left(\L_{\eta_{\sigma}^{\varepsilon}}\right)+\mathfrak{c}_{\beta,A,T}\varepsilon$, we deduce that $Y_{\sigma} \leq \esssup_{\eta\in \mathcal{T}_{[\sigma,T]}}\mathcal{E}^f_{\sigma,\eta}\left(\L_{\eta}\right)$. Then the proof of Theorem \ref{Main result of sec 5} is complete.
\end{proof}

%
%
\appendix
\section*{Appendix}
\label{Appendix}
\section*{BSDEs with default jump and stochastic Lipschitz coefficient}
\label{BSDEs}
\subsection*{Existence, uniqueness and preliminary estimates}
In this section, we focus on investigating the existence and uniqueness of solutions for a specific form of BSDEs with default jump associated with $\left(\xi,f\right)$. Furthermore, we present some comparison theorems for this type of BSDEs.

We consider the following BSDE:
\begin{equation}
	Y_t= \xi+\int_t^T f(s,Y_s,Z_s,U_s)ds-\int_t^T Z_s d B_s -\int_t^T U_s dM_s, \quad 0 \leq t \leq T.
	\label{basic equation : classical BSDE}
\end{equation}

\begin{theorem}[Existence and Uniqueness]
	\emph{}
	Assume that \textbf{(H1)} and \textbf{(H2)} hold for a sufficiently large $\beta>0$. Then, the BSDE \eqref{basic equation : classical BSDE} admit a unique solution $(Y,Z,U)\in \mathfrak{B}^2_{\beta} \times \mathcal{H}^2_{\beta}  \times \mathcal{M}^2_{\gamma,\beta}$.
	\label{Existence and uniqueness unconstrained BSDE}
\end{theorem}

\begin{proposition}
	Let $(Y^j,Z^j,U^j)$ be the unique solution of the RBSDE \eqref{basic equation lower} associated
	with data $(\xi^j,f^j)$ for $j \in \{1,2\}$ satisfying assumptions \textbf{(H1)}, \textbf{(H2)}. Then, for any $\beta >2$, there exists a constant $\mathfrak{c}_{\beta}>0$ such that, $s \in [0,T]$,
	\begin{equation*}
		\begin{split}
			&\mathbb{E}\left[\sup_{s \leq t \leq T}e^{\beta A_t} \left| Y^1_t-Y^2_t \right|^2 \mid \mathcal{F}_s  \right]\\
			& \leq \mathfrak{c}_{\beta}\left(\mathbb{E}\left[e^{\beta A_T}\left| \xi^1-\xi^2\right|^2 \mathcal{F}_s \right]\right.\\
			&\qquad \qquad \left.+\mathbb{E}\left[ \left.\int_{s}^{T}  e^{\beta A_r}\left|\dfrac{f^1\left(r,Y^2_r,Z^2_r,U^2_r\right)-f^2\left(r,Y^2_r,Z^2_r,U^2_r\right) }{\alpha_r}\right|^2 dr \right| \mathcal{F}_s\right]  \right).
		\end{split}
	\end{equation*}
	\label{Preliminalry estimaes}
\end{proposition}
\subsection*{Comparison principals}
Let $(Y^j,Z^j,U^j)$ be the unique solution of the RBSDE \eqref{basic equation lower} associated
with data $(\xi^j,f^j)$ for $j \in \{1,2\}$ satisfying assumptions \textbf{(H1)}, \textbf{(H2)} and such that the condition \textbf{(H4)}. Then we have the following result:
\begin{theorem}[Comparison theorems]
	Assume that:
	\begin{itemize}
		\item $\xi^1 \leq \xi^2$.
		\item $
		f^1\left(t, Y_t^2, Z_t^2,U^2_t\right) \leq f^2\left(t, Y_t^2, Z_t^2,U^2_t\right),~t \in [0,T],~d\mathbb{P}\otimes dt\text{-a.e.}
		$
	\end{itemize}
	Then, we have
	\begin{itemize}
		\item[(i)] Comparison theorem: $Y^1_t \leq Y^2_t$, $\forall t \in [0,T]$ a.s.
		\item[(ii)] Strict comparison theorem: Suppose moreover, that $\psi>-1$ a.s., and $Y^1_{t_0}=Y^2_{t_0}$ a.s. for some $t_0 \in [0,T]$. Then $\xi_1=\xi_2$ a.s., and $Y^1=Y^2$ on $[t_0,T]$. 
	\end{itemize}
	\label{Comparison for BSDEs}
\end{theorem}

We now give a comparison theorem for a kind of BSDEs with {\it generalized} driver.
\begin{theorem}
	Assume that conditions of Theorem \ref{Comparison for BSDEs} hold and that $K^1$ and $K^2$ are two regulated, optional, increasing processes in $\mathcal{S}^2$. If there exist a triplet $\left(Y_t^i, Z_t^i,U^i_t\right)_{0 \leq t \leq T}, i=1,2$, that belongs to $\mathcal{S}^{2,\alpha}_{\beta} \times \mathcal{H}^{2}_{\beta} \times \mathcal{U}^{2}_{\lambda,\beta}$ satisfying the equations
	\begin{equation*}
		Y_t^i=\xi^i+\int_t^T f^i\left(s, Y_s^i, Z_s^i,U^i_s\right) d s+K_T^i-K_t^i-\int_t^T Z_s^i d B_s-\int_{t}^{T}U^i_s dM_s, \quad i=1,2,
		\label{BSDE Comparison}
	\end{equation*}
	and, moreover, if  $K^1-K^2$ is an increasing process, then $Y_t^1 \geq Y_t^2$, $t \in [0,T]$ a.s.
	\label{Comparison theorem for GBSDE}
\end{theorem}
\begin{remark}
	\begin{itemize}
		\item The proof of Theorem \ref{Existence and uniqueness unconstrained BSDE} is constructed in two steps: The first one for a generator independent of the $(y,z,u)$-variables, and it's based on the representation theorem \ref{Representation property} for the existence and uniqueness and a similar procedure as in \cite[Proposition 1]{elmansourielotmani2023} for the integrable property of the solution.\\
		For the second step, we approach the solution of BSDE \eqref{basic equation : classical BSDE} through Picard's iteration method using the previous Step in a suitable Banach space.
		
		\item The proof of Proposition \ref{Preliminalry estimaes} follows a similar argument as the one used in the preliminary estimates showed in  \cite[Proposition 1]{elmansourielotmani2023}.
		
		\item The proof of Theorem \ref{Comparison for BSDEs}-(i) can be obtained from the one of Theorem \ref{Comparison theorem} with Remark \ref{Comparison BSDE}, while the proof of the strict version and of the Theorem \ref{Comparison theorem for GBSDE}  follows similar approach based on the results of Corollary \ref{integration by part formula}, Remark \ref{RCLL integration by part formula}, Remark \ref{Remark on the right-continuous part} and  Proposition \ref{Integr problm solved}.
	\end{itemize}
\end{remark}
\subsection*{It\^o's formula for processes with regulated trajectories}
Using the well-known It\^o's formula for right-continuous semimartingales, we may develop a generic formula that applies to a specific class of irregular processes that are not always right continuous.
\begin{theorem}
	Let $Y=\left(Y^1,Y^2,\cdots,Y^n\right)$ be an adapted $n$-dimensional process with regulated trajectories of the form
	\begin{equation}
		Y_t=Y^{\ast}_t+\sum_{0 \leq s <t} \Delta_{+} Y_s,\quad \forall t \in [0,T],
		\label{form of the process for Ito}
	\end{equation}
	where $Y^{\ast}=\left(Y^{\ast,1},Y^{\ast,2},\cdots,Y^{\ast,n}\right)$ is an RCLL adapted $n$-dimensional semimartingale and $\sum_{s <t} \left|\Delta_{+}  Y_s\right|  <\infty$  a.s. Let $F$ be
	a twice continuously differentiable function on $\mathbb{R}^n$. Then the process
	$\left(F\left(Y_t\right)\right)_{t \leq T}$ also has the form \eqref{form of the process for Ito}. Moreover, almost surely, for each $n \geq 1$ and all $t \leq T$,
	\begin{equation*}
		\begin{split}
			F\left(Y_t\right)
			&=F\left(Y_0\right)+\sum_{k=1}^n \int_{0}^{t} D^k F\left(Y_{s-}\right)dY^{\ast,k}_s+\dfrac{1}{2}\sum_{k,l=1}^n  \int_{0}^{t} D^k D^l F\left(Y_{s-}\right) d\left[Y^{\ast,k},Y^{\ast,l}\right]^c_s\\
			& \quad+\sum_{0 < s \leq t}\left\{ F\left(Y_s\right)-F\left(Y_{s-}\right)-\sum_{k=1}^n  D^k F\left(Y_{s-}\right)\Delta_{-}Y^{\ast,k}_s\right\}\\
			&\quad+\sum_{0 \leq s < t} \left\{F\left(Y_{s+}\right)-F\left(Y_{s}\right)\right\},
		\end{split}
	\end{equation*}
	where $D^k$ denotes the differentiation operator with respect to the $k$-th coordinate,and $\left[\cdot,\cdot\right]^c$ denotes the continuous part of the quadratic variation $\left[\cdot,\cdot\right]$.
	\label{Ito's formula Theorem}
\end{theorem}
\begin{proof}
	See the proof of Theorem A.1 in \cite{klimsiak2019reflected}.
\end{proof}

\begin{corollary}
	Let $n=2$ and $F:\mathbb{R}^2 \to \mathbb{R}$ A twice differential equation on $\mathbb{R}^2$ given by $F(x,y)=e^{\beta x} \left| y\right|^2$. Let $Y^1=A:=\left(A_t\right)_{t \leq T}$ be a continuous adapted process with finite variation on $[0,T]$, and  $Y^2:=\left(Y_t\right)_{t \leq T}$ a one-dimensional adapted process with regulated paths of the form \eqref{form of the process for Ito}. Then, By applying Theorem \ref{Ito's formula Theorem}, we get for all $t \leq T$,
	\begin{equation*}
		\begin{split}
			e^{\beta A_t} \left| Y_t \right|^2&=  \left| Y_0 \right|^2+\beta \int_{0}^{t} e^{\beta A_s} \left| Y_s \right|^2 dA_s+2\int_{0}^{t} e^{\beta A_s} Y_{s-} dY^{\ast}_s+\int_{0}^{t} e^{\beta A_s} d \left[Y^{\ast}\right]^c_s\\
			&\quad +\sum_{0 < s \leq t} e^{\beta A_s} \left| \Delta_{-} Y_{s} \right|^2+\sum_{0 \leq s < t} e^{\beta A_s} \left(\left| \Delta_{+} Y_{s}\right|^2+2Y_{s}\Delta_{+} Y_s \right).
		\end{split}
	\end{equation*}
	\label{Application of Ito formula}
\end{corollary}
As an application of Theorem \ref{Ito's formula Theorem} using the function $F(y^1,y^2)=y^1 y^2$ an the decomposition $\left[Y^{1,\ast},Y^{2,\ast}\right]=\left[Y^{1,\ast},Y^{2,\ast}\right]^c+\sum_{0<s \leq \cdot} \Delta_{-} Y^1_s \Delta_{-} Y^2_s$ (see \cite[Theorem I.4.52]{jacod2013limit}), we obtain the following integration by part formula:
\begin{corollary}
	Let $Y^1$, $Y^2$ be two adapted processes with regulated trajectories of the form \eqref{form of the process for Ito}. Then
	\begin{equation*}
		\begin{split}
			Y^1_t Y^2_t&=Y^1_0 Y^2_0+\int_{0}^{t}Y^1_{s-}dY^{2,\ast}_s+\int_{0}^{t}Y^2_{s-}dY^{1,\ast}_s+\left[Y^{1,\ast},Y^{2,\ast}\right]_t\\
			&\qquad+\sum_{0 \leq s <t}\left(Y^1_{s+}Y^2_{s+}-Y^1_s Y^2_s\right),\quad t \in [0,T].
		\end{split}
	\end{equation*}
	\label{integration by part formula}
\end{corollary}
\begin{remark}
	Note that when $X^1$ and $X^2$ have  RCLL paths, then $X^1=X^{1,\ast}$, $X^2=X^{2,\ast}$, and Corollary \ref{integration by part formula} turns into the classical integration by part formula given, for example, in \cite[Corollary II.2]{Protter}.
	\label{RCLL integration by part formula}
\end{remark}

\subsection*{Tanaka-type formula}
Theorem 66, found on page 210 in \cite{Protter}, is extended in this section. It provides an alternative version of the classical Tanaka's to the case of strong optional semimartingale. We will utilize this lemma to demonstrate the reflected BSDE with one irregular barrier comparison theorem.

First, let us recall the definition of an $\mathbb{F}$-optional semimartingales which can be found in Gal'{\v{c}}uk's seminal work \cite[Page 462]{gal1981optional}.
\begin{definition}
	The process $(\mathsf{X}_t)_{t \leq T}$ is termed an optional semimartingale if it can be expressed as $\mathsf{X} = \mathsf{X}_0 + \mathsf{K} + \mathsf{N}$, where $\mathsf{N}$ is an RCLL local martingale, $\mathsf{K}$ is an $\mathbb{F}$-optional process with finite variation and regulated trajectories, satisfying $\mathsf{K}_0=\mathsf{N}_0=0$, and $\mathsf{X}_0$ is an $\mathcal{F}_0$-measurable finite random variable.
\end{definition}

The key takeaway from this subsection is as follows:
\begin{lemma}[Tanaka-type Formula]
	Let $Y$ be an adapted process of the form \eqref{form of the process for Ito} and $\Phi : \mathbb{R} \rightarrow \mathbb{R}$ be a convex function. Then, $\Phi(Y)$ is an $\mathbb{F}$-optional semimartingale. Moreover, denoting by $\Phi^{\prime}$ the left-hand derivative of the convex function $F$. Then, we have
	$$
	\Phi(Y_t)=\Phi(Y_0)+\int_{0}^{t} \Phi^{\prime}(Y_{s-}) dY^{\ast}_s+\mathcal{L}_t,
	$$
	where $\mathcal{L}$ is a non-decreasing $\mathbb{F}$-adapted process with regulated trajectories (which is in general neither left-continuous nor right-continuous) such that
	$$
	\Delta_{-}\mathcal{L}_{t}=\Delta_{-}\Phi(Y_t)-\Phi^{\prime}(Y_{s-})\Delta_{-}Y_t,\quad \text{ and } \quad\Delta_{+} \mathcal{L}_{t}=\Phi(Y_{t+})-\Phi(Y_t).
	$$
	\label{Tanaka-type Formula}
\end{lemma}
\begin{proof}
	See the proof of Lemma 9.1 in \cite{grigorova2020optimal}.
\end{proof}

\section*{Funding}
No funding was received for this paper.

\section*{Disclosure statement}
No potential conflict of interest was reported by the authors.

%
\end{document}